\documentclass[sn-mathphys,Numbered]{sn-jnl}

\usepackage{graphicx}%
\usepackage{multirow}%
\usepackage{amsmath,amssymb,amsfonts}%
\usepackage{amsthm}%
\usepackage{mathrsfs}%
\usepackage[title]{appendix}%
\usepackage{xcolor}%
\usepackage{textcomp}%
\usepackage{manyfoot}%
\usepackage{booktabs}%
\usepackage{algorithm}%
\usepackage{algorithmicx}%
\usepackage{algpseudocode}%
\usepackage{listings}%

\theoremstyle{thmstyleone}%
\newtheorem{theorem}{Theorem}
\newtheorem{proposition}[theorem]{Proposition}%

\theoremstyle{thmstyletwo}%
\newtheorem{remark}{Remark}%

\theoremstyle{thmstylethree}%

\begin{document}

\title[L.I.G.S.J.S]{A K$\ddot{A}$HLERIAN APPROACHE TO  THE   SCHR$\ddot{O}$DINGER EQUATION IN SIEGEL-JACOBI SPACE OF THE LOGNORMAL DISTRIBUTION.}

\author*[1]{\fnm{} \sur{Prosper Rosaire Mama  Assandje }}\email{mamarosaire@yahoo.fr, mamarosaire@facsciences-uy1.cm}
\author[2]{\fnm{} \sur{Joseph Dongho }}\email{josephdongho@yahoo.fr}
\equalcont{These authors contributed equally to this work.}

\author[3]{\fnm{} \sur{ Thomas Bouetou Bouetou}}\email{tbouetou.@gmail.com}
\equalcont{These authors contributed equally to this work.}

 \affil*[1]{\orgdiv{Department of Mathematics}, \orgname{University of Yaounde1},
\orgaddress{\street{usrectorat@.univ-yaounde1.cm}, \city{Yaounde},
\postcode{337}, \state{Center}, \country{Cameroon}}}

\affil[2]{\orgdiv{Department of Mathematics and Computer Science},
\orgname{University of Maroua},
\orgaddress{\street{decanat@fs.univ-maroua.cm}, \city{Maroua},
\postcode{814}, \state{Far -North}, \country{Cameroon}}}


\affil[3]{\orgdiv{Computer Engineering Department.}, \orgname{Higher
national school of Polytechnic},
\orgaddress{\street{thomas.bouetou@polytechnique.cm}, \city{Yaounde
I}, \postcode{8390}, \state{Center}, \country{Cameroon}}}

\abstract{In this paper, we describe the evolution of spectral
curves in the Siegel-Jacobi space through the Schr$\ddot{o}$dinger
equation constructed from a K$\ddot{a}$hler geometry induced on the
lognormal statistical manifold via Dombrowski's construction. We
introduce new holomorphic structures and show that the Hamiltonian
vector field coincides with the fundamental vector field generated
by holomorphic isometries. We construct the time-dependent
Schr$\ddot{o}$dinger equation from this geometric setting and show
that the associated energy is not constant, but varies with time.
This work establishes a bridge between K$\ddot{a}$hler geometry,
statistical models, and the formalism of quantum mechanics.}




\maketitle

\section{Introduction}\label{sec1}
The Schr$\ddot{o}$dinger equation stands as a cornerstone of quantum
mechanics, governing the time evolution of physical systems at the
microscopic scale\cite{erw}. In recent decades, a growing body of
work has suggested deep connections between this quantum framework
and differential geometry - particularly K$\ddot{a}$hler geometry,
which naturally arises from the geometric structure of statistical
models. In \cite{molitor1,molitor2,molitor3} it has recently been
suggested that the quantum formalism could be based on
K$\ddot{a}$hler geometry,  which follows naturally from statistics.
In this work, we aim to explore the dynamics of spectral curves
associated with the lognormal statistical manifold, using the
Dombrowski construction to endow its tangent bundle with a
K$\ddot{a}$hler structure. The geometric setting of our analysis is
the Siegel-Jacobi space, which elegantly combines symplectic and
complex features and plays a prominent role in the geometric
formulation of quantum mechanics. This framework leads us to the
following central question: How does the K$\ddot{a}$hler geometry
derived from a Dombrowski construction on the lognormal statistical
manifold enable the formulation of a Schr$\ddot{o}$dinger equation
in the Siegel-Jacobi space, and what are the geometric and dynamical
consequences for the evolution of the associated spectral curves? To
address this, we investigate several subproblems:  How can a
K$\ddot{a}$hler structure be constructed on the lognormal manifold
via the Dombrowski formalism?  What is the nature of the holomorphic
and isometric functions that preserve this structure?  How can the
evolution of spectral curves in the Siegel-Jacobi space be governed
by a Schr$\ddot{o}$dinger-type equation?  Is the resulting energy
function time-independent, or does it reveal a deeper dynamical
behavior? Work in this field goes back to Bohr, who devised a
radical model of the atom with electrons orbiting a nucleus. He
established that electrons can only be in special orbits. All other
orbits are not possible. He shows that electrons do not radiate
energy as long as they remain in their defined or stationary orbit.
It shows that energy radiation in the form of quanta only takes
place when an electron jumps from a higher-energy authorised orbit
to another lower-energy authorised orbit. In $1926$,
Schr$\ddot{o}$dinger\cite{erw} became interested in De Broglie's
matter wave. He studied the type of equation obeyed by this De
Broglie wave and developed a wave equation, now called the
fundamental equation of quantum mechanics, which governs the
microscopic world. Many authors have stressed the importance of
K$\ddot{a}$hler geometry in relation to the quantum formalism
\cite{molitor1,molitor2,molitor3,cire1,cire2,hes}. In
\cite{ash,dom},it is known that a quantum system, with Hilbert space
$\mathbb{C}^{n}$ , can be entirely described by means of the
K$\ddot{a}$hler structure of $\mathbb{P}(\mathbb{C}^{n})$; this is
the so-called geometrical formulation of quantum mechanics. On the
bases mentioned by Mathieu Molitor in \cite{dom}, the quantum
formalism has a link with the geometric theory of information. In
this case, the formalism should be rewritten and the role of the
above statistical-K$\ddot{a}$hler geometry should be fully
clarified. Mathieu Molitor in \cite{dom}, sets out two hypotheses in
his paper: Firstly, the quantum formalism has indeed an
information-theoretical origin. In this case, the formalism should
be rewritten and the role of the above statistical-K$\ddot{a}$hler
geometry should be fully clarified. See \cite{dom, chir, cli}.
Secondly, quantum mechanics cannot be derived from
information-theoretical principles. In this case, one should still
explain the relationship between the above definition of $Spec(f)$,
which is a priori independent of representation theory, and the
definition of the spectrum of an operator. It may well be that there
is some geometrical content hidden behind the main results of
functional analysis that goes beyond the well-known correspondence
between the space of K$\ddot{a}$hler functions of the complex
projective space and the space of Hermitian operators. It is in this
order that we show that, given $S$ be the statistical manifold
associated with the lognormal
 distribution, endowed with the Fisher information metric $h$.
  Consider the tangent bundle $TS$ and let $T(TS)$ be its tangent bundle.
   The Dombrowski construction applied to $(S, h)$ induces an almost
    Hermitian structure $(g, J, \omega)$ on $TS$, where:
 $g$ is the natural extension of $h$ to $TS$,
  $\omega$ is the fundamental 2-form associated to $g$ and $J$,
     $J$ is defined by $J(A, B, C) = (A, -C, B)$, for any $A(x), B(x), C(x) \in T_{p}S$.
Then $(TS, g, J, \omega)$ is a K$\ddot{a}$hler manifold. We show
that: In the coordinates system
$\left(\theta_{1},\theta_{2},\dot{\theta}_{1},\dot{\theta}_{2}\right)$
the K$\ddot{a}$hler structure $(g, J, \omega)$ is given by
\begin{equation*}
    g=\left(
        \begin{array}{cccc}
          -\frac{1}{2\theta_{2}} & \frac{\theta_{1}}{2\theta^{2}_{2}} & 0 & 0 \\
          \frac{\theta_{1}}{2\theta^{2}_{2}} & -\frac{\theta^{2}_{1}-\theta_{2}}{2\theta^{3}_{2}} & 0 & 0 \\
          0 & 0 &  -\frac{1}{2\theta_{2}} &  \frac{\theta_{1}}{2\theta^{2}_{2}} \\
0& 0 &  \frac{\theta_{1}}{2\theta^{2}_{2}} & -\frac{\theta^{2}_{1}-\theta_{2}}{2\theta^{3}_{2}} \\
        \end{array}
    \right),\; J=\left(
        \begin{array}{cccc}
         0&0 & -1 & 0 \\
          0 &0 & 0 & -1 \\
          1 & 0 &  0 &  0 \\
0& 1 & 0 &0 \\
        \end{array}
    \right),\end{equation*}\; \begin{equation*}\omega=\left(
        \begin{array}{cccc}
         0&0 & -\frac{1}{2\theta_{2}}& \frac{\theta_{1}}{2\theta^{2}_{2}} \\
          0 &0 & \frac{\theta_{1}}{2\theta^{2}_{2}} & -\frac{\theta^{2}_{1}-\theta_{2}}{2\theta^{3}_{2}} \\
          \frac{1}{2\theta_{2}} & -\frac{\theta_{1}}{2\theta^{2}_{2}} &  0 &  0 \\
-\frac{\theta_{1}}{2\theta^{2}_{2}}& \frac{\theta^{2}_{1}-\theta_{2}}{2\theta^{3}_{2}} & 0 &0 \\
        \end{array}
    \right).
\end{equation*}
     In the coordinates system
$\left(\eta_{1},\eta_{2},\dot{\theta}_{1},\dot{\theta}_{2}\right)$,
the K$\ddot{a}$hler structure $(g, J, \omega)$ is  given by
\begin{equation*}
    g=\left(
        \begin{array}{cccc}
          2\theta_{1}^{2}-2\theta_{2}  & 2\theta_{1}\theta_{2} & 0 & 0 \\
          2\theta_{1}\theta_{2} & 2\theta_{2}^{2} & 0 & 0 \\
          0 & 0 &  -\frac{1}{2\theta_{2}} &  \frac{\theta_{1}}{2\theta^{2}_{2}} \\
0& 0 &  \frac{\theta_{1}}{2\theta^{2}_{2}} & -\frac{\theta^{2}_{1}-\theta_{2}}{2\theta^{3}_{2}} \\
        \end{array}
    \right),\; J=\left(
        \begin{array}{cccc}
         0&0 & \frac{1}{2\theta_{2}} & -\frac{\theta_{1}}{2\theta^{2}_{2}} \\
          0 &0 & -\frac{\theta_{1}}{2\theta^{2}_{2}} & \frac{\theta^{2}_{1}-\theta_{2}}{2\theta^{3}_{2}} \\
          2\theta_{1}^{2}-2\theta_{2}  & 2\theta_{1}\theta_{2} &  0 &  0 \\
2\theta_{1}\theta_{2}& 2\theta_{2}^{2}  & 0 &0 \\
        \end{array}
    \right),\end{equation*}\; \begin{equation*}\omega=\left(
        \begin{array}{cccc}
         0&0 & 1 & 0 \\
          0 &0 & 0 & 1 \\
          -1 & 0 &  0 &  0 \\
0& -1 & 0 &0 \\
        \end{array}
    \right)
\end{equation*}
In the same we show that, the Christoffel's coefficients on the
k$\ddot{a}$hler manifold define by the lognormal distribution is
given by\begin{eqnarray*}
 \left(\Gamma^{g}_{11} \right)^{1} &=& \left(\Gamma^{g}_{11} \right)^{2}
 =\left(\Gamma^{g}_{12} \right)^{1}=\left(\Gamma^{g}_{12} \right)^{2}=\left(\Gamma^{g}_{21} \right)^{1}
  =\left(\Gamma^{g}_{21} \right)^{2}=\left(\Gamma^{g}_{22} \right)^{1}=\left(\Gamma^{g}_{22}
  \right)^{2}=0.
\end{eqnarray*} with $\left(\Gamma^{g}_{ij} \right)^{k}$ is the
Christoffel's coefficient for the K$\ddot{a}$hler metric. We show
that  all holomorphic isometric function  $\varphi$ on $TS$ whose
elements satisfy $\varphi^{*}g=g$, and $\varphi_{*}J=J\varphi_{*}$
is given by
\[\displaystyle {\it \varphi(\theta,\dot{\theta})}\, := \, \left( \theta_{{1}}+{\it k_{1}},\theta_{{2}}+{\it k_{2}},\dot{\theta}_{{1}}+{\it k_{3}},\dot{\theta}_{{2}}+{\it k_{4}}\\
\mbox{} \right) \] with $k_{1},\;k_{2},\;k_{3},\;k_{4}\in\mathrm{I\!
R}$ and $\varphi: TS\longrightarrow TS$ be a diffeomorphism. We show
that, For all  $L\in \left\{F,G,H,H,P,Q\in\mathfrak{g}^{J}\right\}$
the Hamiltonian vector field $X_{J}L$ on Siegel-Jacobi space is
given by
\begin{eqnarray*}
  X_{\psi(F)}(\theta,\dot{\theta})&=&\left(0,0,0,0\right),\;\;\;\;\; X_{\psi(P)}(\theta,\dot{\theta})=\left(2\theta_{1}^{2}-2\theta_{2},2\theta_{1}\theta_{2},-2\theta_{1}\theta_{2},-2\theta_{2}^{2}\right)\\
  X_{\psi(G)}(\theta,\dot{\theta})&=&\left(2\theta_{1}^{2}-2\theta_{2} +2\theta_{1}\theta_{2}\dot{\theta}_{2},2\theta_{1}\theta_{2}+2\theta_{2}^{2}\dot{\theta}_{2},-2\theta_{1}\theta_{2}(\theta_{2}+1),-2\theta_{2}^{2}(\theta_{2}+1)\right),\\
  X_{\psi(H)}(\theta,\dot{\theta})&=&\left(0,0,-2\theta_{1}^{2}+2\theta_{2} ,-2\theta_{1}\theta_{2}\right),\;\;\;\;\;X_{\psi(R)}(\theta,\dot{\theta})=\left(0,0,0,0\right)\\
  X_{\psi(Q)}(\theta,\dot{\theta})&=&\left(0,0,-2\theta_{1}\theta_{2},-2\theta_{2}^{2} \right)
\end{eqnarray*}
 We show that there exist $\alpha:I \longrightarrow \mathbb{H}\times
\mathbb{C}=\mathbb{S}^{J},\; \alpha(t)\longmapsto(-iz_{2},\;z_{1})$
an integral curve of the Hamiltonian vector field $X_{J}L$, and
constant application
$K_{L}\Phi(t):=\kappa\left(\alpha(t),L\right)=cste$, with
$t=(z_{1},\;z_{2}), \; I=\mathbb{C}\times i\mathbb{H}$, and satisfy
the following equation on Jacobi space $\mathbb{S}^{J}$ by given
\begin{eqnarray*}
    \frac{id
\Phi(t)}{dt}&=&\frac{1}{2}\left(-\lambda_{1}\log^{2}(x)-\lambda_{2}\left(2iz_{1}+\left(-iz_{2}+2iz_{1}\log(x)\right)^{2}\right)\right)\Phi(t)\nonumber\\
&+&\frac{1}{2}\left(\lambda_{3}\left(2\log(x)z_{2}-4z_{1}\log^{2}(x)+i\right)+\lambda_{4}\left(-z_{2}+2z_{1}\log(x)\right)\Phi(t)\right)\\
&+&\frac{1}{2}\left(\xi(\theta,\dot{\theta},x)\log(x)-\frac{1}{4}\lambda_{5}\right)\Phi(t)+\frac{1}{2}\left(\lambda_{1}K_{F}\Phi(t)
+\lambda_{2}K_{G}\Phi(t)+\lambda_{3}K_{H}\Phi(t)\right.\nonumber\\\\
&+&\left.\lambda_{4}K_{P}\Phi(t)+\xi(\theta,\dot{\theta},x)K_{Q}\Phi(t)+\lambda_{5}K_{R}\Phi(t)\right)\nonumber\\
\end{eqnarray*}
We show that if the unique application
$K_{L}\Phi(t):=\kappa\left(\alpha(t),L\right)$ is zero, then the
Hamiltonian equation  on Jacobi space $\mathbb{S}^{J}$ is
 given by
\begin{eqnarray*}
    \frac{id
\Phi(t)}{dt}&=&\frac{1}{2}\mathcal{H}(\theta,\dot{\theta},x)\Phi(t)\end{eqnarray*}
with
\begin{eqnarray*}
    \mathcal{H}(\theta,\dot{\theta},x)&=&-\lambda_{1}\log^{2}(x)-\lambda_{2}\left(2iz_{1}
    +\left(-iz_{2}+2iz_{1}\log(x)\right)^{2}\right)\\
&+&\lambda_{3}\left(2\log(x)z_{2}-4z_{1}\log^{2}(x)+i\right)+\lambda_{4}\left(-z_{2}+2z_{1}\log(x)\right)\\
&+&\xi(\theta,\dot{\theta},x)\log(x)-\frac{1}{4}\lambda_{5}.
\end{eqnarray*}
We show that, the  Shr$\ddot{o}$dinger equation on Jacobi space
$\mathbb{S}^{J}$ if only if $\xi(\theta,\dot{\theta},x)=\gamma, \;
\gamma\in \mathrm{I\! R}$, and we have \begin{eqnarray*}
    \frac{id
\Phi(t)}{dt}&=&\frac{1}{2}Q(L)\Phi(t)+\frac{1}{2}K(L)\Phi(t).
\end{eqnarray*}, and furthermore, if $K_{L}\Phi(t):=\kappa\left(\alpha(t),L\right)$ is
zero, then the  Shr$\ddot{o}$dinger equation becomes
\begin{eqnarray*}
    \frac{id
\Phi(t)}{dt}&=&\frac{1}{2}\mathcal{H}\Phi(t)\end{eqnarray*} with
$\mathcal{H}=\frac{1}{2}Q(L)$, and the Hamiltonian is given by
\begin{eqnarray*}
    \mathcal{H}(\theta,\dot{\theta},x)&=&-\lambda_{1}\log^{2}(x)-\lambda_{2}\left(2iz_{1}
    +\left(-iz_{2}+2iz_{1}\log(x)\right)^{2}\right)\\
&+&\lambda_{3}\left(2\log(x)z_{2}-4z_{1}\log^{2}(x)+i\right)+\lambda_{4}\left(-z_{2}+2z_{1}\log(x)\right)\\
&+&\gamma\log(x)-\frac{1}{4}\lambda_{5}.
\end{eqnarray*}
After the introduction, the first section recall the preliminaries
motion on theory of statistical manifold, in section $2$ we
 present the Dombrowski's construction on Amari's $\ell$-representation on lognormal
 manifold, in section $3$, we construct the  K$\ddot{a}$hler function, and holomorphic isometric function. At the
 end, in section $4$,we present Shr$\ddot{o}$dinger equation on Jacobi space.

\section{Preliminaries}\label{sec2}
Let $S = \left\{p_{ \theta}(x),\left.
                        \begin{array}{ll}
                         \theta\in \Theta & \hbox{} \\
                          x\in \mathcal{X} & \hbox{}
                        \end{array}
                      \right.
\right\}$ be the set of probabilities $p_{ \theta}$, parameterized
by $ \Theta$, open subset of $\mathrm{I\! R}^{n}$; on the sample
space $\mathcal{X}\subseteq\mathrm{I\! R}$. Let
$\mathcal{F}{(\mathcal{X},\mathrm{I\! R})}$ be the space of
real-value smooth functions on $\mathcal{X}$. According to Ovidiu
\cite{ovidiu}, the log-likelihood function is a mapping defined by:
\begin{eqnarray}
l:S&\longrightarrow& \mathcal{F}{(\mathcal{X},\mathrm{I\! R})}\nonumber\\
p_{ \theta} &\longmapsto&  l\left(p_{ \theta}\right)(x) = \log
p_{\theta}(x).\nonumber
\end{eqnarray}
Sometimes, for convenient reasons, this will be denoted by
$l(x,\theta)=l\left(p_{ \theta}\right)(x)$. Let
$\theta=\left(\theta_{1},\dots,\theta_{n}\right)$ be a coordinate
system on $S$. Let $h$ be a Riemannian structure on $S$. Any $X,Y\in
\mathfrak{X}(S)$ induce a  $C^{\infty}(S)$-linear form $h_{X}$ on
$\mathfrak{X}(S)$ defined by $h_{X}(Y)=h(X,Y)$. That is for all
$X\in\mathfrak{X}(S)$, $h_{X}\in \Omega(S) $. Where
$\mathfrak{X}(S)$ is the module of vector fields over  $S$.
Therefore, each Riemannian structure $h$ on $S$ induces an
isomorphism
\begin{eqnarray*}
b_{h}:\mathfrak{X}(S)&\tilde{\longrightarrow}&
\Omega(S)\\
X&\longmapsto& h_{X}.
\end{eqnarray*}
 Let
 $\Phi\in C^{\infty}(S)$, we have
 $d\Phi\in \Omega(S)$ and there exists $X_{\Phi}\in \mathfrak{X}(S)$ such that $b_{h}(X_{\Phi})=d\Phi$.
As a linear mapping, $b_{h}$ has a matrix $[b_h]$  with respect to
the basis pair $\left(\partial_{\theta_{i}},\; d\theta_{i}\right)$
and we know that $(h_{ij})_{1\leq i,j\leq n}=[b_h]$. According to
Amari \cite{Shu}, we have \begin{equation} [b_h]=(h_{ij})_{1\leq
i;j\leq
n}=\left(-\mathbb{E}[\partial_{\theta_{i}}\partial_{\theta_{j}}l(x,\theta)]\right).\label{eq1}
\end{equation}

\subsection{Affine connection and statistical manifold}\label{subsec1}
The collection of all $T_{p}S$ defined a $2n-$dimensional manifold
$T(S)$ called Tangent bundle of $S$. More explicitly, $T(S)$ is
completely defined by the projection mapping $\pi:TS\longrightarrow
S$ such that $\pi^{-1}(\{p\}) \cong\mathrm{I\! R}^{n}$ and $\pi$
have the usual local trivialization property. $A:S\longrightarrow
TS$ is a vector field on $S$ if
 $\pi\circ A=1_{S}$.
 The collection $\mathfrak{X}(S)$ of all vector fields on $S$ is an $C^{\infty}(S)-$module. An affine connection is
an $\mathrm{I\! R}$-bilinear mapping\\
$ \nabla:\mathfrak{X}(S)\longrightarrow\mathfrak{X}(S):
(A,B)\mapsto\nabla_{A}B $
such that:\\
 $\nabla_{ A} (fB )= (Af) (B)  + f\nabla_{ A} (B),\;
 \nabla_{ fA} (B )= f\nabla_{ A} (B).$\\
In local coordinate
$\theta=\left(\theta_{1},\dots,\theta_{n}\right)$ of $S$ , $\nabla_{
\partial_{\theta_{i}}} (\partial_{\theta_{j}}
)=\sum\Gamma_{ij}^{k}\partial_{\theta_{i}}$ and $\Gamma_{ij}^{k}$ is
Christoffel coefficient given by: $ \Gamma_{ij}^{k}=
\frac{1}{2}h^{km}[\partial_{\theta_{i}}.h_{jm}+
\partial_{\theta_{j}}.h_{ik} - \partial_{\theta_{k}}.h_{ij}]$.\\
In \cite{lauri}, a statistical manifold $(S,h,\nabla)$ to be a
manifold $S$ equipped with metric $h$ and connection $\nabla$ such
that
\begin{enumerate}
    \item [i)]$\nabla$ is torsion-free
    \item [ii)]$\nabla_{h}\equiv C$ is totally symmetric
\end{enumerate}
 Equivalently, a manifold has a statistical structure
when the conjugate (with respect to $g$) $\nabla^{*}$ of a
torsion-free connection $\nabla$ is also torsion-free.
\subsection{Dombrowski's construction.}\label{subsec2}
According to Mathieu Molitor \cite{dom},
 it is denote by $\pi : TS \longrightarrow S$ the canonical projection and
by $K$ the connector associated to $\nabla$. Recall that $K$ is the
unique map $T (TS) \longrightarrow TS$ satisfying  $\nabla_{X}Y =
KY\ast X$ for all vector fields $X, Y$ on $S$ (here $Y\ast X$
denotes the derivative of $Y$ in the direction of $X$). Given $u_{p}
\in T_{p}S$, the subspaces
\begin{equation}\label{eq1}
    \left\{
                              \begin{array}{ll}
                                Hor(TS)u_{p} := \left\{ Z \in T_{u_{p}} (TS) / KZ =
0\right\}
 & \hbox{} \\
                                Ver(TS)u_{p}  := \left\{ Z \in T_{u_{p}} (TS) / \pi\ast_{u_{p}}Z =
0\right\}& \hbox{}
                              \end{array}
                            \right.
\end{equation}
 are respectively called the space of horizontal tangent
vectors and the space of vertical tangent vectors of $TS$ at
$u_{p}$. They are both isomorphic to $T_{p}S$ in a natural way, and
led to the following decomposition:
\begin{equation}\label{eq2}
    T_{u_{p}} (TS) \cong Hor(TS)u_{p}\oplus  Hor(TS)u_{p}\cong T_{p}S
\oplus T_{p}S.
\end{equation}
  More generally, $\nabla$  determines an isomorphism of vector
bundles over $S$:
\begin{equation}\label{eq3}
    T (TS) \cong TS \oplus TS \oplus TS,
\end{equation} the isomorphism being

\begin{equation}\label{eq4}
    T_{u_{p}} (TS)\ni A_{u_{p}}\mapsto \left(u_{p},\pi\ast_{u_{p}}\ast
A_{u_{p}},KA_{u_{p}}\right).
\end{equation}
 If there is no danger of confusion, we shall thus regard an
element of $T_{u_{p}} (TS)$ as a triple
$\left(u_{p},v_{p},w_{p}\right) $, where $u_{p},v_{p},w_{p}\in
T_{p}S$. The second component $v_{p}$ is usually referred to as the
horizontal component with respect to $\nabla$ and $w_{p}$ the
vertical component. Let $h$ be a Riemannian metric on $S$. Together
with $\nabla$, the couple $(h,\nabla)$ determines an almost
Hermitian structure on $TS$ we have the following formulas:
\begin{equation}\label{eq5}
    \left\{
      \begin{array}{ll}
        g_{u_{p}}\left(\left(u_{p},v_{p},w_{p}\right),\left(u_{p},\bar{v}_{p},\bar{w}_{p}\right)\right):=h_{p}\left(v_{p},\bar{v}_{p}\right)+h_{p}\left(w_{p},\bar{w}_{p}\right) & \hbox{(metric)} \\
        \omega_{u_{p}}\left(\left(u_{p},v_{p},w_{p}\right),\left(u_{p},\bar{v}_{p},\bar{w}_{p}\right)\right):=h_{p}\left(v_{p},\bar{w}_{p}\right)-h_{p}\left(w_{p},\bar{v}_{p}\right) & \hbox{(2-form)} \\
        J_{u_{p}}\left(u_{p},v_{p},w_{p}\right):=\left(u_{p},-w_{p},v_{p}\right),\; (a) & \hbox{(almost complex structure)}
      \end{array}
    \right.
\end{equation}
where $u_{p},v_{p},w_{p},\bar{v}_{p},\bar{w}_{p}\in T_{p}S$.
Clearly, $J^{2} = -Id$ and $g(J . , J . ) = g( . , . )$, which means
that $(TS, g, J)$ is an almost Hermitian manifold, and one readily
sees that $g,\; J$ and $w$ are compatible, i.e., that $\omega=
g\left(  J . , . \right)$ . The 2-form $\omega$ is thus the
fundamental 2-form of the almost Hermitian manifold $(TS, g, J)$:
This is Dombrowski's construction.\\Let $(h,\nabla,\nabla^{*})$ be a
dualistic structure on $S$ and $(g, J, \omega)$ the almost Hermitian
structure on $TS$ associated to $(h,\nabla)$ according to
Dombrowski's construction. According to Mathieu Molitor \cite{dom},
\begin{equation}\label{eq6}
    (TS, g, J, \omega)\; \textrm{is K$\ddot{a}$hler}\; \Leftrightarrow \;(S,
h,\nabla,\nabla^{*}) \;\textrm{is dually flat}.
\end{equation}
The same,
\begin{enumerate}
    \item In the coordinates $(\theta_{i}, \dot{\theta}_{i})$, we
    have the following structure
    \begin{equation}\label{eq7}
        g=\left(
            \begin{array}{cc}
              h_{ij} & 0 \\
            0 & h_{ij}\\
            \end{array} \right),\; J=\left(
            \begin{array}{cc}
              0 & I \\
 I & 0\\
            \end{array} \right),\;\omega=\left(
            \begin{array}{cc}
              0 &  h_{ij} \\
 - h_{ij} & 0 \\
            \end{array}\right),\; i,j\in\{1,\dots,n\}.
    \end{equation}
    \item In the coordinates $(\eta_{i}, \dot{\theta}_{i})$, we
    have the following structure
\begin{equation}\label{eq06}
        g=\left(
            \begin{array}{cc}
              h^{ij} & 0 \\
 0 & h_{ij} \\
            \end{array}\right),\; J=\left(
            \begin{array}{cc}
              0 & -h_{ij}  \\
 h^{ij}  & 0 \\
            \end{array}\right),\;\omega=\left(
            \begin{array}{cc}
              0 & I \\
 - I & 0 \\
            \end{array}\right),\; i,j\in\{1,\dots,n\}.
    \end{equation}
\end{enumerate}
 Let $f :
TS\longrightarrow \mathrm{I\! R}$ be a smooth function. Given an
affine coordinate system $x: U\longrightarrow \mathrm{I\! R}^{n}$
with respect to $\nabla$ on $S$, we have the following equivalence:
$f$ is K$\ddot{a}$hler on $\pi^{-1}(U)$ if and only if
\begin{equation}\label{eq8}
    \left\{
      \begin{array}{ll}
        \frac{\partial^{2}f}{\partial x_{i}\partial x_{j}}-
         \frac{\partial^{2}f}{\partial \dot{x}_{i}\partial \dot{x}_{j}}= 2\sum_{b=1}^{n}\left(\Gamma^{h}_{ij} \right)^{b}\circ \pi\frac{\partial f}{\partial x_{b}}& \hbox{;} \\
        \frac{\partial^{2}f}{\partial x_{i}\partial \dot{x}_{j}}-
         \frac{\partial^{2}f}{\partial x_{j}\partial \dot{x}_{i}}= 2\sum_{b=1}^{n}\left(\Gamma^{h}_{ij} \right)^{b}\circ \pi\frac{\partial f}{\partial \dot{x}_{b}}& \hbox{;}
      \end{array}
    \right.
\end{equation}
\begin{remark}\cite{dom}\label{remark10} Let $\varphi: TS\longrightarrow TS$ be a
diffeomorphism. In the coordinates system $(\theta, \dot{\theta})$
we can be written $\varphi(\theta, \dot{\theta})=
\left(\varphi^{1}(\theta, \dot{\theta}),\varphi^{2}(\theta,
\dot{\theta}),\varphi^{3}(\theta, \dot{\theta}),\varphi^{4}(\theta,
\dot{\theta})\right)$ ,
\begin{eqnarray*}
  \varphi\; \textrm{is holomorphic} &\Leftrightarrow&
  \varphi^{1}+i\varphi^{3}\;
  \textrm{and}\;
 \varphi^{2}+i\varphi^{4}\; \textrm{are holomorphic functions} \\
  &\Leftrightarrow& \frac{\partial}{\partial
  \bar{z}_{k}}\left(\varphi^{1}+i\varphi^{3}\right)=\frac{\partial}{\partial
  \bar{z}_{k}}\left(\varphi^{2}+i\varphi^{4}\right)=0,
  k\in\{0,\dots,n\}
\end{eqnarray*}
where $\varphi^{2}<0,\; \frac{\partial}{\partial
  \bar{z}_{k}}=\frac{1}{2}\left\{\frac{\partial}{\partial
  \theta_{k}}+i\frac{\partial}{\partial
  \dot{\theta}_{k}}\right\}$.\\

According to Mathieu Molitor \cite{dom},  if $\varphi$ is an
isometry, then
\begin{eqnarray}\label{lem4}
  \frac{\partial\varphi^{2}}{\partial
  \theta_{1}} &=& \frac{\partial\varphi^{2}}{\partial
  \dot{\theta}_{1}},\; \left(\frac{\partial\varphi^{2}}{\partial
  \theta_{2}}\right)^{2}+\left(\frac{\partial\varphi^{4}}{\partial
  \theta_{2}}\right)^{2}=\left(\frac{\partial\varphi^{2}}{\partial
  \theta_{2}}\right)^{2}\\
  \frac{\partial\varphi^{4}}{\partial
  \theta_{1}} &=& \frac{\partial\varphi^{4}}{\partial
  \dot{\theta}_{1}},\; \left(\frac{\partial\varphi^{2}}{\partial
  \dot{\theta_{2}}}\right)^{2}+\left(\frac{\partial\varphi^{4}}{\partial
  \dot{\theta}_{2}}\right)^{2}=\left(\frac{\partial\varphi^{2}}{\partial
  \theta_{2}}\right)^{2}\nonumber
\end{eqnarray}
\end{remark}

\section{ Dombrowski's construction on Amari's $\ell$-representation on lognormal manifold.}\label{sec3}
Let $S = \left\{p_{ \theta}(x),\left.
                        \begin{array}{ll}
                         \theta\in \Theta & \hbox{} \\
                          x\in \mathcal{X} & \hbox{}
                        \end{array}
                      \right.
\right\}$ be the set of probabilities $p_{ \theta}$, parameterized
by $ \Theta$, open subset of $\mathbb{R}^{n}$; on the sample space
$\mathcal{X}\subseteq\mathbb{R}$.
\begin{proposition}\label{proaa}
 Let $\left(S,h,\nabla\right)$ the statistical manifold endowed with the Fisher
information metric $h$. Let $T_{p}S$ the tangent vector on $S$ at
$p$. Consider the tangent bundle $TS$ and let $T(TS)$ be its tangent
bundle. The Dombrowski construction applied to $(S, h)$ induces an
almost Hermitian structure $(g, J, \omega)$ on $TS$, where:
\begin{itemize}
    \item $g$ is the natural extension of $h$ to $TS$,
    \item $\omega$ is the fundamental 2-form associated to $g$ and $J$,
    \item $J$ is defined by $J(u_{p}, v, w) = (u_{p}, -w, v)$, for any $u_{p}, v, w \in T_{p}S$ with $v\in Hor(TS)_{u_{p}}$,
    and $w\in Ver(TS)_{u_{p}}$.
\end{itemize} Then $(TS, g, J, \omega)$ is a K$\ddot{a}$hler
manifold.
 \end{proposition}

\begin{proof}
By Dombrowski's construction, the metric $g$, almost complex
structure $J$, and fundamental form $\omega$ on $TS$ are given by:
\begin{equation*}
g((u_{p}, v, w), (u_{p}, \bar{v}, \bar{w})) = h(v, \bar{v}) + h(w,
\bar{w})\end{equation*}
\begin{equation*}
\omega((u_{p}, v, w), (u_{p}, \bar{v}, \bar{w})) = h(v, \bar{w}) -
h(w, \bar{v})\end{equation*}
\begin{equation*}
J(u_{p}, v, w) = (u_{p}, -w, v).\end{equation*} We verify that
$J^2=-Id$ as follows, \begin{equation*}J^2(u, v, w)=J(J(u, v,
w))=J(u_{p}, -v, w) =(u_{p}, -v, -w)= -(u_{p}, v,
w)=-Id\end{equation*}, and that
\begin{equation*}g\left(J(u_{p}, v, w), J(u_{p}, \bar{v}, \bar{w})\right) =g\left((u_{p}, -w, v),
J(u_{p},-\bar{w}, \bar{v} )\right)=h(-w, -\bar{w})+h(v,
\bar{v}),\end{equation*} we know that $h(-w, -\bar{w})=h(w,
\bar{w})$; So we have \begin{equation*}g\left((J(u_{p}, v, w),
J(u_{p}, \bar{v}, \bar{w})\right)=g\left((u_{p}, v, w), (u_{p},
\bar{v}, \bar{w})\right)\end{equation*}, $J$ preserves the inner
product defined by $g$, the structure $(TS, g, J)$ is almost
Hermitian. which confirms that $(TS, g, J)$ is an almost Hermitian
manifold. Additionally, the fundamental 2-form $\omega$ is closed,
i.e., $d\omega = 0$, and the connection $\nabla$ is flat, implying
that $J$ is integrable. Hence, all conditions are satisfied for
$(TS, g, J, \omega)$ to be a K$\ddot{a}$hler.
\end{proof}

Let $S = \left\{p_{ \theta}(x)=\frac{1}{\sqrt{2\pi}\sigma
x}e^{-\frac{(\log x-\mu)^{2}}{2\sigma^{2}}},\left.
                        \begin{array}{ll}
                         \theta= (\mu , \sigma )\in \mathrm{I\! R}\times \mathrm{I\! R}^{*}_{+} & \hbox{} \\
                          x\in \mathrm{I\! R}^{*}_{+} & \hbox{}
                        \end{array}
                      \right.
\right\}$ be a lognormal model where $p_{ \theta}$ is its density
function. $
    \theta_{2}= -\frac{1}{2\sigma^{2}},\;\theta_{1}=
\frac{\mu}{\sigma^{2}},\;-2\theta_{2}>0. $
So,\begin{eqnarray}\label{eq16} \ell(x, \theta)=-\log
x+\theta_{1}\log(x)+\theta_{2}\log^{2}(x)
+\frac{\theta^{2}_{1}}{4\theta_{2}}\nonumber\\
+\frac{1}{2}\log(-2\theta_{2})-\frac{1}{2}\log(2\pi)
\end{eqnarray}
with the potential function
\begin{eqnarray}\label{eq17}
\Phi(\theta)&=&-
\frac{\theta^{2}_{1}}{4\theta_{2}}-\frac{1}{2}\log(-2\theta_{2})+\frac{1}{2}\log(2\pi)
\end{eqnarray}

Let $T^{\ell}=\left\{u_{p}|u_{p}=A^{1}
\partial_{\theta_{1}}\ell(x,\theta)+A^{2}\partial_{\theta_{2}}\ell(x,\theta)\right\}$ the Amari's $\ell$-representation on
lognormal manifold in the natural basis, with
$\partial_{\theta_{i}}:=\frac{\partial}{\partial \theta_{i}}$
\begin{equation*}B^{\ell}=\left\{\partial_{\theta_{1}}\ell(x,\theta)=\log(x)+\frac{\theta_{1}}{2\theta_{2}},\partial_{\theta_{2}}\ell(x,\theta)=\log^{2}(x)
-\frac{\theta^{2}_{1}}{4\theta^{2}_{2}}\right\}.\end{equation*} We
denote that $T^{\ell}\cong T_{p}S$, so
$\partial_{\theta_{1}}\ell(x,\theta),\partial_{\theta_{2}}\ell(x,\theta)\in
T_{p}S$. The bundle tangent on $S$ is given by
\begin{equation}\label{eq13}
    TS=\coprod_{p\in S}T_{p}S=\coprod_{p\in S}\{p\}\times T_{p}S=\left\{\left(p,u_{p}\right)|p\in S\; and \; u_{p}\in T_{p}S\right\}.
\end{equation}

\begin{theorem}\label{pro2}Let $h$ the  Riemannian metric,  and
$\nabla$ the affine connection on $S$. The manifold $\left(TS, g, J,
\omega\right)$ with respect to the K$\ddot{a}$hler structure $(g, J,
\omega)$ is a K$\ddot{a}$hler manifold on $TS$ associated with
$(h,\nabla)$.
\begin{enumerate}
    \item In the coordinates system
$\left(\theta_{1},\theta_{2},\dot{\theta}_{1},\dot{\theta}_{2}\right)$
the K$\ddot{a}$hler structure $(g, J, \omega)$ is given by
\begin{equation}\label{eq14}
    g=\left(
        \begin{array}{cccc}
          -\frac{1}{2\theta_{2}} & \frac{\theta_{1}}{2\theta^{2}_{2}} & 0 & 0 \\
          \frac{\theta_{1}}{2\theta^{2}_{2}} & -\frac{\theta^{2}_{1}-\theta_{2}}{2\theta^{3}_{2}} & 0 & 0 \\
          0 & 0 &  -\frac{1}{2\theta_{2}} &  \frac{\theta_{1}}{2\theta^{2}_{2}} \\
0& 0 &  \frac{\theta_{1}}{2\theta^{2}_{2}} & -\frac{\theta^{2}_{1}-\theta_{2}}{2\theta^{3}_{2}} \\
        \end{array}
    \right),\; J=\left(
        \begin{array}{cccc}
         0&0 & -1 & 0 \\
          0 &0 & 0 & -1 \\
          1 & 0 &  0 &  0 \\
0& 1 & 0 &0 \\
        \end{array}
    \right),\end{equation}\; \begin{equation*}\omega=\left(
        \begin{array}{cccc}
         0&0 & -\frac{1}{2\theta_{2}}& \frac{\theta_{1}}{2\theta^{2}_{2}} \\
          0 &0 & \frac{\theta_{1}}{2\theta^{2}_{2}} & -\frac{\theta^{2}_{1}-\theta_{2}}{2\theta^{3}_{2}} \\
          \frac{1}{2\theta_{2}} & -\frac{\theta_{1}}{2\theta^{2}_{2}} &  0 &  0 \\
-\frac{\theta_{1}}{2\theta^{2}_{2}}& \frac{\theta^{2}_{1}-\theta_{2}}{2\theta^{3}_{2}} & 0 &0 \\
        \end{array}
    \right)
\end{equation*}
    \item In the coordinates system
$\left(\eta_{1},\eta_{2},\dot{\theta}_{1},\dot{\theta}_{2}\right)$,
the K$\ddot{a}$hler structure $(g, J, \omega)$ is  given by
\begin{equation}\label{eq15}
    g=\left(
        \begin{array}{cccc}
          2\theta_{1}^{2}-2\theta_{2}  & 2\theta_{1}\theta_{2} & 0 & 0 \\
          2\theta_{1}\theta_{2} & 2\theta_{2}^{2} & 0 & 0 \\
          0 & 0 &  -\frac{1}{2\theta_{2}} &  \frac{\theta_{1}}{2\theta^{2}_{2}} \\
0& 0 &  \frac{\theta_{1}}{2\theta^{2}_{2}} & -\frac{\theta^{2}_{1}-\theta_{2}}{2\theta^{3}_{2}} \\
        \end{array}
    \right),\; J=\left(
        \begin{array}{cccc}
         0&0 & \frac{1}{2\theta_{2}} & -\frac{\theta_{1}}{2\theta^{2}_{2}} \\
          0 &0 & -\frac{\theta_{1}}{2\theta^{2}_{2}} & \frac{\theta^{2}_{1}-\theta_{2}}{2\theta^{3}_{2}} \\
          2\theta_{1}^{2}-2\theta_{2}  & 2\theta_{1}\theta_{2} &  0 &  0 \\
2\theta_{1}\theta_{2}& 2\theta_{2}^{2}  & 0 &0 \\
        \end{array}
    \right),\end{equation}\; \begin{equation*}\omega=\left(
        \begin{array}{cccc}
         0&0 & 1 & 0 \\
          0 &0 & 0 & 1 \\
          -1 & 0 &  0 &  0 \\
0& -1 & 0 &0 \\
        \end{array}
    \right)
\end{equation*}
\end{enumerate}
\end{theorem}

\begin{proof}
Let $S = \left\{p_{ \theta}(x)=\frac{1}{\sqrt{2\pi}\sigma
x}e^{-\frac{(\log x-\mu)^{2}}{2\sigma^{2}}},\left.
                        \begin{array}{ll}
                         \theta= (\mu , \sigma )\in \mathrm{I\! R}\times \mathrm{I\! R}^{*}_{+} & \hbox{} \\
                          x\in \mathrm{I\! R}^{*}_{+} & \hbox{}
                        \end{array}
                      \right.
\right\}$ be a lognormal model where $p_{ \theta}$ is its density
function.

It follow from (\ref{eq17}), from theorem 4 in \cite{mama} and
Amari's theorem 3.5 and 3.4 in \cite{Shu}, we have
\begin{equation}\label{re1}
 \eta_{i}= \frac{\partial \Phi(\theta)}{\partial \theta_{i} }
 ,\;1\leq i \leq 2,\;\textrm{so we have}\; \eta_{1}= -\frac{\theta_{1}}{2\theta_{2}},\;\eta_{2}=\frac{\theta^{2}_{1}-2\theta_{2}}{4\theta^{2}_{2}}
\end{equation}
We define,\begin{equation}\label{gij1}h_{ij}=\langle
\partial_{\theta_{i}},\partial_{\theta_{j}}\rangle_{h}:=h
\left(\partial_{\theta_{i}},\partial_{\theta_{j}}\right).
\end{equation}
Then we can use the Amari's \cite{Shu}, $l-$representation
associated respectively to coordinate  $\theta$ and $\eta$; and
define the following scalar production on $T_{\theta}$ where
$T_{\theta}$ is the $l-$representation of the tangent space with
respect to the coordinate $\theta$. Using (\ref{eq1}) we have:
\[
 h\left(\partial_{\theta_{1}}\ell(x,\theta),\partial_{\eta_{1}}\ell(x,\theta)\right)
=\mathbb{E}\left[\partial_{\eta_{1}}\left(-\frac{\theta_{1}}{2\theta_{2}}\right)\right]
=\mathbb{E}\left[1\right]=1.\]
\[
h\left(\partial_{\theta_{1}}\ell(x,\theta),\partial_{\eta_{2}}\ell(x,\theta)\right)=\mathbb{E}\left[\partial_{\eta_{2}}\left(-\frac{\theta_{1}}{2\theta_{2}}\right)\right]=0.\]
\[h\left(\partial_{\theta_{2}}\ell(x,\theta),\partial_{\eta_{1}}\ell(x,\theta)\right)=\mathbb{E}\left[\partial_{\eta_{1}}\left(\frac{\theta^{2}_{1}-2\theta_{2}}{4\theta^{2}_{2}}\right)\right]=0.
  \]
  \[h\left(\partial_{\theta_{2}}\ell(x,\theta),\partial_{\eta_{2}}\ell(x,\theta)\right)=\mathbb{E}\left[\partial_{\eta_{2}}\left(\frac{\theta^{2}_{1}-2\theta_{2}}{4\theta^{2}_{2}}\right)\right]=1.
  \]
Two coordinate system $\left(\theta_{i}\right)_{1\leq i \leq 2}$ and
$\left(\eta_{i}\right)_{1\leq i \leq 2}$ are dual mutually.\\
Using (\ref{gij1}), and (\ref{eq1}), we obtain
\[
h_{11}=-\mathbb{E}\left[\partial_{\theta_{1}}\partial_{\theta_{1}}\ell(x,\theta)\right]
=-\frac{1}{2\theta_{2}}.\]
\[
h_{12}=-\mathbb{E}\left[\partial_{\theta_{1}}\partial_{\theta_{2}}\ell(x,\theta)\right]
=\frac{\theta_{1}}{2\theta^{2}_{2}}.\]
\[h_{21}=-\mathbb{E}\left[\partial_{\theta_{2}}\partial_{\theta_{1}}\ell(x,\theta)\right]
=\frac{\theta_{1}}{2\theta^{2}_{2}}\]
  \[h_{22}=-\mathbb{E}\left[\partial_{\theta_{2}}\partial_{\theta_{2}}\ell(x,\theta)\right]
=-\frac{\theta^{2}_{1}-\theta_{2}}{2\theta^{3}_{2}}
\] So we have the tenseur metric given by
\begin{eqnarray*}
h&=&\left[\begin{array}{cc}
-\frac{1}{2\theta_{2}} &\frac{\theta_{1}}{2\theta^{2}_{2}} \\
 \frac{\theta_{1}}{2\theta^{2}_{2}}&-\frac{\theta^{2}_{1}-\theta_{2}}{2\theta^{3}_{2}}
 \end{array}\right]
\end{eqnarray*}
and  it is straightforward that for $\theta_{2}\neq 0$, we obtain
\begin{eqnarray*} h^{-1} &=& \left[ \begin{array}{cc} 2\theta_{1}^{2}-2\theta_{2}      &   2\theta_{1}\theta_{2}\\ 2\theta_{1}\theta_{2}  &  2\theta_{2}^{2}  \\ \end{array}
 \right].
\end{eqnarray*}
We define  the canonical projection
\begin{equation}\label{pi}
\pi : TS \longrightarrow S.\end{equation} Let
$\left(\theta_{1},\theta_{2}\right)$  be a coordinate system on $S$.
So, $\left(\theta_{1}\circ \pi,\theta_{2}\circ \pi,d\theta_{1},
d\theta_{2}\right)$ be a coordinate system on $TS$. In the follows
we denote $\left(\theta_{1}\circ \pi,\theta_{2}\circ
\pi,d\theta_{1}, d\theta_{2}\right)$ as
$\left(\theta_{1},\theta_{2},\dot{\theta}_{1},
\dot{\theta}_{2}\right)$. We define an other canonical projection
 \begin{equation}\label{pit}
 \pi_{TS}:T (TS)
\longrightarrow TS,\end{equation} and we have $\left(\theta_{1}\circ
\pi\circ \pi_{TS},\theta_{2}\circ \pi\circ
\pi_{TS},(d\theta_{1})\circ \pi_{TS}, (d\theta_{2})\circ
\pi_{TS},d(\theta_{1}\circ \pi),d(\theta_{2}\circ
\pi),d(d\theta_{1}),d(d\theta_{2})\right)$ the local coordinates
system on $T(TS)$. By setting\begin{eqnarray*}
  m_{1}&=& \theta_{1}\circ
\pi\circ \pi_{TS},\; m_{2} = \theta_{2}\circ \pi\circ
\pi_{TS},d\theta_{1}\circ
\pi_{TS} \\
   u_{1} &=& (d\theta_{1})\circ \pi_{TS},\;
u_{2} = (d\theta_{2})\circ
\pi_{TS} \\
   v_{1} &=& d(\theta_{1}\circ \pi),\;
   v_{2} = d(\theta_{2}\circ \pi) \\
t_{1} &=& d(d\theta_{1}),\;
   t_{2} = d(d\theta_{2})
\end{eqnarray*}
By definition of the reduction of canonical projection $K$ and
$\pi_{*}$ in \cite{dom}, we have\begin{eqnarray*}
  K(m_{1},u_{1},v_{1},t_{1}) &=& (\theta_{1}\circ
\pi\circ \pi_{TS}, d(d\theta_{1}))\\
  K(m_{2},u_{2},v_{2},t_{2}) &=& (\theta_{2}\circ
\pi\circ \pi_{TS}, d(d\theta_{2})) \\
  \pi_{*}(m_{1},u_{1},v_{1},t_{1}) &=& (\theta_{1}\circ
\pi\circ \pi_{TS},d(\theta_{1}\circ \pi) )\\
  \pi_{*}(m_{2},u_{2},v_{2},t_{2}) &=& (\theta_{2}\circ
\pi\circ \pi_{TS},d(\theta_{2}\circ \pi) ).
\end{eqnarray*}
In follows we write
$u_{p}=\left(\theta_{1},\theta_{2},\dot{\theta}_{1},\dot{\theta}_{2}\right)$
the coordinate system on $TS$, and  we also write
$\left(\eta_{1},\eta_{2},\dot{\theta}_{1},\dot{\theta}_{2}\right)$.\\
In this basis, the tangent vector $A$ on bundle tangent $ T(TS)$ in
$u_{p}$ is given by,
\begin{eqnarray*}
A=a^{1}\partial_{\theta_{1}}+a^{2}\partial_{\theta_{2}}
+b^{1}\partial_{\dot{\theta}_{1}}+b^{2}\partial_{\dot{\theta}_{2}},\;
where,\;a^{i},b^{i}\in C^{\infty}(S)
\end{eqnarray*}
with \begin{eqnarray*}
v_{p}=a^{1}\partial_{\theta_{1}}+a^{2}\partial_{\theta_{2}}\in
Hor(TS)_{u_{p}}, and
\;w_{p}=b^{1}\partial_{\dot{\theta}_{1}}+b^{2}\partial_{\dot{\theta}_{2}}\in
Ver(TS)_{u_{p}}
\end{eqnarray*}
we know that $ g\left(v_{p},w_{p}\right)=0 $ more
explicitly,\begin{eqnarray*}
g\left(\partial_{\theta_{1}},\partial_{\dot{\theta}_{1}}\right)
=g\left(\partial_{\theta_{1}},\partial_{\dot{\theta}_{2}}\right)=g\left(\partial_{\dot{\theta}_{1}},\partial_{\theta_{1}}\right)=
g\left(\partial_{\dot{\theta}_{2}},\partial_{\theta_{1}}\right)=0
\end{eqnarray*}
and
\begin{eqnarray*}
g\left(\partial_{\dot{\theta}_{i}},\partial_{\dot{\theta}_{j}}\right)
&=&h\left(\partial_{\dot{\theta}_{i}},\partial_{\dot{\theta}_{j}}\right),\;g\left(\partial_{\theta_{i}},\partial_{\theta_{j}}\right)
=h\left(\partial_{\theta_{i}},\partial_{\theta_{j}}\right)
\end{eqnarray*}
So, for
\begin{eqnarray*}
\bar{A}=\bar{a}^{1}\partial_{\theta_{1}}+\bar{a}^{2}\partial_{\theta_{2}}
+\bar{b}^{1}\partial_{\dot{\theta}_{1}}+\bar{b}^{2}\partial_{\dot{\theta}_{2}},\;
\in T_{u_{p}}(T(S))
\end{eqnarray*}
and \begin{eqnarray*}
\bar{v}_{p}=\bar{a}^{1}\partial_{\theta_{1}}+\bar{a}^{2}\partial_{\theta_{2}}\in
Hor(TS)_{u_{p}},\;
\bar{w}_{p}=\bar{b}^{1}\partial_{\dot{\theta}_{1}}+\bar{b}^{2}\partial_{\dot{\theta}_{2}}\in
Ver(TS)_{u_{p}}
\end{eqnarray*} we write
\begin{eqnarray*}g\left(A,\bar{A}\right)&=&a^{1}\bar{a}^{1}g\left(\partial_{\theta_{1}},\partial_{\theta_{1}}\right)
+a^{1}\bar{a}^{2}g\left(\partial_{\theta_{1}},\partial_{\theta_{2}}\right)+
a^{1}\bar{b}^{1}g\left(\partial_{\theta_{1}},\partial_{\dot{\theta}_{1}}\right)+
a^{1}\bar{b}^{2}g\left(\partial_{\theta_{1}},\partial_{\dot{\theta}_{2}}\right)\\
&&+a^{2}\bar{a}^{1}g\left(\partial_{\theta_{2}},\partial_{\theta_{1}}\right)
+a^{2}\bar{a}^{2}g\left(\partial_{\theta_{2}},\partial_{\theta_{2}}\right)+
a^{2}\bar{b}^{1}g\left(\partial_{\theta_{2}},\partial_{\dot{\theta}_{1}}\right)+
a^{2}\bar{b}^{2}g\left(\partial_{\theta_{2}},\partial_{\dot{\theta}_{2}}\right)\\
&&+b^{1}\bar{a}^{1}g\left(\partial_{\dot{\theta}_{1}},\partial_{\theta_{1}}\right)
+b^{1}\bar{a}^{2}h\left(\partial_{\dot{\theta}_{1}},\partial_{\theta_{2}}\right)+
b^{1}\bar{b}^{1}g\left(\partial_{\dot{\theta}_{1}},\partial_{\dot{\theta}_{1}}\right)+
b^{1}\bar{b}^{2}g\left(\partial_{\dot{\theta}_{1}},\partial_{\dot{\theta}_{2}}\right)\\
&&++b^{2}\bar{a}^{1}g\left(\partial_{\dot{\theta}_{2}},\partial_{\theta_{1}}\right)
+b^{2}\bar{a}^{2}g\left(\partial_{\dot{\theta}_{2}},\partial_{\theta_{2}}\right)+
b^{2}\bar{b}^{1}g\left(\partial_{\dot{\theta}_{2}},\partial_{\dot{\theta}_{1}}\right)+
b^{2}\bar{b}^{2}g\left(\partial_{\dot{\theta}_{2}},\partial_{\dot{\theta}_{2}}\right)
\end{eqnarray*}
we obtain
\begin{eqnarray*}g\left(A,\bar{A}\right)&=&a^{1}\bar{a}^{1}g\left(\partial_{\theta_{1}},\partial_{\theta_{1}}\right)
+a^{1}\bar{a}^{2}g\left(\partial_{\theta_{1}},\partial_{\theta_{2}}\right)+a^{2}\bar{a}^{1}g\left(\partial_{\theta_{2}},\partial_{\theta_{1}}\right)
+a^{2}\bar{a}^{2}g\left(\partial_{\theta_{2}},\partial_{\theta_{2}}\right)\\
&&+b^{1}\bar{b}^{1}g\left(\partial_{\dot{\theta}_{1}},\partial_{\dot{\theta}_{1}}\right)+
b^{1}\bar{b}^{2}g\left(\partial_{\dot{\theta}_{1}},\partial_{\dot{\theta}_{2}}\right)+b^{2}\bar{b}^{1}g\left(\partial_{\dot{\theta}_{2}},\partial_{\dot{\theta}_{1}}\right)+
b^{2}\bar{b}^{2}g\left(\partial_{\dot{\theta}_{2}},\partial_{\dot{\theta}_{2}}\right)
\end{eqnarray*}
\begin{eqnarray*}g\left(A,\bar{A}\right)&=&a^{1}\bar{a}^{1}h_{11}
+a^{1}\bar{a}^{2}h_{12}+a^{2}\bar{a}^{1}h_{21}
+a^{2}\bar{a}^{2}h_{22}\\
&&+b^{1}\bar{b}^{1}h_{11}+
b^{1}\bar{b}^{2}h_{12}+b^{2}\bar{b}^{1}h_{21}+
b^{2}\bar{b}^{2}h_{22}
\end{eqnarray*}
\begin{eqnarray*}g\left(A,\bar{A}\right)&=&h\left(v_{p},\bar{v}_{p}\right)
+h\left(w_{p},\bar{w}_{p}\right)
\end{eqnarray*}
a tangent vector $A$ on bundle tangent $ T(TS)$ in $u_{p}$ is
identify by $\left(u_{p},v_{p},w_{p}\right)$. The same $B$ is
identify by $\left(u_{p},\bar{v}_{p},\bar{w}_{p}\right)$ we obtain
the Dombrowski's definition
\begin{eqnarray*}g\left(\left(u_{p},v_{p},w_{p}\right),\left(u_{p},\bar{v}_{p},\bar{w}_{p}\right)\right)&=&h\left(v_{p},\bar{v}_{p}\right)
+h\left(w_{p},\bar{w}_{p}\right)
\end{eqnarray*}
the matrix form is given by
\begin{equation*}
    g=\left(
        \begin{array}{cccc}
          h_{11} & h_{12} & 0 & 0 \\
          h_{21} & h_{22} & 0 & 0 \\
          0 & 0 &  h_{11} &  h_{12} \\
0& 0 &  h_{21} & h_{22} \\
        \end{array}
    \right)
    \end{equation*}we obtain
\begin{equation*}
    g=\left(
        \begin{array}{cccc}
          -\frac{1}{2\theta_{2}} & \frac{\theta_{1}}{2\theta^{2}_{2}} & 0 & 0 \\
          \frac{\theta_{1}}{2\theta^{2}_{2}} & -\frac{\theta^{2}_{1}-\theta_{2}}{2\theta^{3}_{2}} & 0 & 0 \\
          0 & 0 &  -\frac{1}{2\theta_{2}} &  \frac{\theta_{1}}{2\theta^{2}_{2}} \\
0& 0 &  \frac{\theta_{1}}{2\theta^{2}_{2}} & -\frac{\theta^{2}_{1}-\theta_{2}}{2\theta^{3}_{2}} \\
        \end{array}
    \right)
    \end{equation*}
 We written $J$ by
\begin{eqnarray*}
J(\partial_{\theta_{1}})&=&\partial_{\dot{\theta}_{1}},\;J(\partial_{\theta_{2}})=\partial_{\dot{\theta}_{2}}\\
J(\partial_{\dot{\theta}_{1}})&=&-\partial_{\theta_{1}},\;J(\partial_{\dot{\theta}_{2}})=-\partial_{\theta_{2}}.
\end{eqnarray*}

We have $J=\left(
        \begin{array}{cccc}
         0&0 & -1 & 0 \\
          0 &0 & 0 & -1 \\
          1 & 0 &  0 &  0 \\
0& 1 & 0 &0 \\
        \end{array}
    \right)$
We verify that $J^{2}=-Id$. $J$ is a complex structure.\\
Therefore,
\begin{eqnarray*}
\omega&=&\frac{1}{2}\omega_{11}d\theta_{1}\wedge
d\theta_{1}+\frac{1}{2}\omega_{12}d\theta_{1}\wedge d\theta_{2}+
+\frac{1}{2}\omega_{13}d\theta_{1}\wedge d\dot{\theta}_{1}+
\frac{1}{2}\omega_{14}d\theta_{1}\wedge d\dot{\theta}_{2}+\\
&&\frac{1}{2}\omega_{21}d\theta_{2}\wedge
d\theta_{1}+\frac{1}{2}\omega_{32}d\dot{\theta}_{1}\wedge
d\theta_{2} +\frac{1}{2}\omega_{43}d\dot{\theta}_{2}\wedge
d\dot{\theta}_{1}+ \frac{1}{2}\omega_{44}d\dot{\theta}_{2}\wedge
d\dot{\theta}_{2}+\\
&&\frac{1}{2}\omega_{22}d\theta_{2}\wedge
d\theta_{2}+\frac{1}{2}\omega_{33}d\dot{\theta}_{1}\wedge
d\dot{\theta}_{1} +\frac{1}{2}\omega_{42}d\dot{\theta}_{2}\wedge
d\theta_{2}+ \frac{1}{2}\omega_{44}d\dot{\theta}_{2}\wedge
d\dot{\theta}_{2}+\\
&&\frac{1}{2}\omega_{31}d\dot{\theta}_{2}\wedge
d\theta_{1}+\frac{1}{2}\omega_{34}d\dot{\theta}_{1}\wedge
d\dot{\theta}_{2} +\frac{1}{2}\omega_{41}d\dot{\theta}_{2}\wedge
d\theta_{1}+ \frac{1}{2}\omega_{23}d\theta_{3}\wedge
d\dot{\theta}_{1}+\\
&&\frac{1}{2}\omega_{24}d\theta_{2}\wedge d\dot{\theta}_{2}
\end{eqnarray*}
So, we calculate $\omega_{ij},\; 1\leq i,j\leq 4$ as follows
$\omega_{ij}=gJ$. We have
\begin{eqnarray*}
\omega_{13}&=&\omega\left(\partial_{\theta_{1}},\partial_{\dot{\theta}_{1}}\right)\\
&=&g(J\partial_{\theta_{1}},\partial_{\dot{\theta}_{1}})\\
&=&g(\partial_{\dot{\theta}_{1}},\partial_{\dot{\theta}_{1}})\\
&=&g_{33}\\
&=&-\frac{1}{2\theta_{2}}
\end{eqnarray*}
\begin{eqnarray*}
\omega_{14}&=&\omega\left(\partial_{\theta_{1}},\partial_{\dot{\theta}_{2}}\right)\\
&=&g(J\partial_{\theta_{1}},\partial_{\dot{\theta}_{2}})\\
&=&g(\partial_{\dot{\theta}_{1}},\partial_{\dot{\theta}_{2}})\\
&=&g_{34}\\
&=&\frac{\theta_{1}}{2\theta^{2}_{2}}
\end{eqnarray*}
\begin{eqnarray*}
\omega_{23}&=&\omega\left(\partial_{\theta_{2}},\partial_{\dot{\theta}_{1}}\right)\\
&=&g(J\partial_{\theta_{2}},\partial_{\dot{\theta}_{1}})\\
&=&g(\partial_{\dot{\theta}_{2}},\partial_{\dot{\theta}_{1}})\\
&=&g_{43}\\
&=&\frac{\theta_{1}}{2\theta^{2}_{2}}
\end{eqnarray*}

\begin{eqnarray*}
\omega_{24}&=&\omega\left(\partial_{\theta_{2}},\partial_{\dot{\theta}_{2}}\right)\\
&=&g(J\partial_{\theta_{2}},\partial_{\dot{\theta}_{2}})\\
&=&g(\partial_{\dot{\theta}_{2}},\partial_{\dot{\theta}_{2}})\\
&=&g_{44}\\
&=&-\frac{\theta^{2}_{1}-\theta_{2}}{2\theta^{3}_{2}}
\end{eqnarray*}
\begin{eqnarray*}
\omega_{31}&=&\omega\left(\partial_{\dot{\theta}_{1}},\partial_{\theta_{1}}\right)\\
&=&g(J\partial_{\dot{\theta}_{1}},\partial_{\theta_{1}})\\
&=&g(-\partial_{\theta_{1}},\partial_{\theta_{1}})\\
&=&-g_{11}\\
&=&\frac{1}{2\theta_{2}}
\end{eqnarray*}
\begin{eqnarray*}
\omega_{41}&=&\omega\left(\partial_{\dot{\theta}_{2}}\ell(x,\theta),\partial_{\theta_{1}}\right)\\
&=&g(J\partial_{\dot{\theta}_{2}},\partial_{\theta_{1}})\\
&=&g(-\partial_{\theta_{2}},\partial_{\theta_{1}})\\
&=&-g_{21}\\
&=&-\frac{\theta_{1}}{2\theta^{2}_{2}}
\end{eqnarray*}
\begin{eqnarray*}
\omega_{32}&=&\omega\left(\partial_{\dot{\theta}_{1}},\partial_{\theta_{2}}\right)\\
&=&g(J\partial_{\dot{\theta}_{1}},\partial_{\theta_{2}})\\
&=&g(-\partial_{\theta_{1}},\partial_{\theta_{2}})\\
&=&-g_{12}\\
&=&-\frac{\theta_{1}}{2\theta^{2}_{2}}
\end{eqnarray*}

\begin{eqnarray*}
\omega_{42}&=&\omega\left(\partial_{\dot{\theta}_{2}},\partial_{\theta_{2}}\right)\\
&=&g(J\partial_{\dot{\theta}_{2}},\partial_{\theta_{2}})\\
&=&g(-\partial_{\theta_{2}},\partial_{\theta_{2}})\\
&=&-g_{22}\\
&=&\frac{\theta^{2}_{1}-\theta_{2}}{2\theta^{3}_{2}}
\end{eqnarray*}
\begin{eqnarray*}
\omega_{11}&=&\omega_{12}=\omega_{21}=\omega_{22}=\omega_{33}=\omega_{34}=\omega_{43}=\omega_{44}=0.
\end{eqnarray*}
We obtain $\omega=\left(
        \begin{array}{cccc}
         0&0 & -\frac{1}{2\theta_{2}}& \frac{\theta_{1}}{2\theta^{2}_{2}} \\
          0 &0 & \frac{\theta_{1}}{2\theta^{2}_{2}} & -\frac{\theta^{2}_{1}-\theta_{2}}{2\theta^{3}_{2}} \\
          \frac{1}{2\theta_{2}} & -\frac{\theta_{1}}{2\theta^{2}_{2}} &  0 &  0 \\
-\frac{\theta_{1}}{2\theta^{2}_{2}}& \frac{\theta^{2}_{1}-\theta_{2}}{2\theta^{3}_{2}} & 0 &0 \\
        \end{array}
    \right)$\\ In the same, $g$ is symmetric
structure, and $J$ skew structure so, the product given the skew
structure. We have
\begin{equation*}
    \left(
        \begin{array}{cccc}
          -\frac{1}{2\theta_{2}} & \frac{\theta_{1}}{2\theta^{2}_{2}} & 0 & 0 \\
          \frac{\theta_{1}}{2\theta^{2}_{2}} & -\frac{\theta^{2}_{1}-\theta_{2}}{2\theta^{3}_{2}} & 0 & 0 \\
          0 & 0 &  -\frac{1}{2\theta_{2}} &  \frac{\theta_{1}}{2\theta^{2}_{2}} \\
0& 0 &  \frac{\theta_{1}}{2\theta^{2}_{2}} & -\frac{\theta^{2}_{1}-\theta_{2}}{2\theta^{3}_{2}} \\
        \end{array}
    \right)\left(
        \begin{array}{cccc}
         0&0 & -1 & 0 \\
          0 &0 & 0 & -1 \\
          1 & 0 &  0 &  0 \\
0& 1 & 0 &0 \\
        \end{array}
    \right)=\left(
        \begin{array}{cccc}
         0&0 & -\frac{1}{2\theta_{2}}& \frac{\theta_{1}}{2\theta^{2}_{2}} \\
          0 &0 & \frac{\theta_{1}}{2\theta^{2}_{2}} & -\frac{\theta^{2}_{1}-\theta_{2}}{2\theta^{3}_{2}} \\
          \frac{1}{2\theta_{2}} & -\frac{\theta_{1}}{2\theta^{2}_{2}} &  0 &  0 \\
-\frac{\theta_{1}}{2\theta^{2}_{2}}& \frac{\theta^{2}_{1}-\theta_{2}}{2\theta^{3}_{2}} & 0 &0 \\
        \end{array}
    \right)
\end{equation*}
Moreover, $\omega^{T}=-\omega.$ So, $\omega$ is skew. $\omega $ is
not degenerate because $det \omega= \frac{1}{4\theta^{6}_{2}}\neq
0$.\\
So, we have\begin{equation*}
d\omega=\Sigma_{i<j<k}\left(\partial_{\theta_{k}}\omega_{ij}+\partial_{\theta_{i}}\omega_{jk}+\partial_{\theta_{j}}\omega_{ki}\right)dx_{k}\wedge
x_{i}\wedge x_{j},\; \forall \;i,j,k\in\left\{1,2,3,4\right\}
\end{equation*}
 We know that $d\omega=0$ if only if
\begin{equation*}
\partial_{x_{k}}\omega_{ij}+\partial_{x_{i}}\omega_{jk}+\partial_{x_{j}}\omega_{ki}=0,\;
\forall \;i,j,k\in\left\{1,2,3,4\right\}
\end{equation*}
with
$x_{1}=\theta_{1},\;x_{2}=\theta_{2},\;x_{3}=\dot{\theta}_{1},\;x_{4}=\dot{\theta}_{2}.$

In this case \begin{equation*}
\partial_{x_{k}}\omega_{ij}+\partial_{x_{i}}\omega_{jk}+\partial_{x_{j}}\omega_{ki}=0,\;
\forall \;i,j,k\in\left\{1,2,3,4\right\}
\end{equation*}
So, $d\omega=0$.\\
 Furthermore, $\nabla$ is  flat, so
the almost complex structure $J$ is integrable. In the coordinates
system
$\left(\eta_{1},\eta_{2},\dot{\theta}_{1},\dot{\theta}_{2}\right)$,
we  have $J^{2}=-Id$.\\
Consider
$X=\left(\eta_{1},\eta_{2},\dot{\theta}_{1},\dot{\theta}_{2}\right)^{T}$,
and show that  $g\left(JX,JY\right)=g\left(X,Y\right)$. We have
$g\left(JX,JY\right)=\left(JX\right)^{T}g\left(JY\right)$.So, we
obtain  $g\left(JX,JY\right)=X^{T}J^{T}gJY$.
 Show that
$J^{T}gJ=g$.
\begin{eqnarray*}
J^{T}gJ&=&\left(
        \begin{array}{cccc}
         0&0 &  2\theta_{1}^{2}-2\theta_{2}  & 2\theta_{1}\theta_{2} \\
          0 &0 & 2\theta_{1}\theta_{2}& 2\theta_{2}^{2}  \\
      \frac{1}{2\theta_{2}} & -\frac{\theta_{1}}{2\theta^{2}_{2}}    &  0 &  0 \\
-\frac{\theta_{1}}{2\theta^{2}_{2}} & \frac{\theta^{2}_{1}-\theta_{2}}{2\theta^{3}_{2}} & 0 &0 \\
        \end{array}
    \right)\left(
        \begin{array}{cccc}
          2\theta_{1}^{2}-2\theta_{2}  & 2\theta_{1}\theta_{2} & 0 & 0 \\
          2\theta_{1}\theta_{2} & 2\theta_{2}^{2} & 0 & 0 \\
          0 & 0 &  -\frac{1}{2\theta_{2}} &  \frac{\theta_{1}}{2\theta^{2}_{2}} \\
0& 0 &  \frac{\theta_{1}}{2\theta^{2}_{2}} & -\frac{\theta^{2}_{1}-\theta_{2}}{2\theta^{3}_{2}} \\
        \end{array}
    \right)\times\\
    &&\left(
        \begin{array}{cccc}
         0&0 & \frac{1}{2\theta_{2}} & -\frac{\theta_{1}}{2\theta^{2}_{2}} \\
          0 &0 & -\frac{\theta_{1}}{2\theta^{2}_{2}} & \frac{\theta^{2}_{1}-\theta_{2}}{2\theta^{3}_{2}} \\
          2\theta_{1}^{2}-2\theta_{2}  & 2\theta_{1}\theta_{2} &  0 &  0 \\
2\theta_{1}\theta_{2}& 2\theta_{2}^{2}  & 0 &0 \\
        \end{array}
    \right)
\end{eqnarray*}
we obtain
\begin{equation*}
J^{T}gJ=\left(
        \begin{array}{cccc}
          2\theta_{1}^{2}-2\theta_{2}  & 2\theta_{1}\theta_{2} & 0 & 0 \\
          2\theta_{1}\theta_{2} & 2\theta_{2}^{2} & 0 & 0 \\
          0 & 0 &  -\frac{1}{2\theta_{2}} &  \frac{\theta_{1}}{2\theta^{2}_{2}} \\
0& 0 &  \frac{\theta_{1}}{2\theta^{2}_{2}} & -\frac{\theta^{2}_{1}-\theta_{2}}{2\theta^{3}_{2}} \\
        \end{array}
    \right)
\end{equation*}We conclude that
$g\left(JX,JY\right)=g\left(X,Y\right)$.\\ Therefore, we have
$d\omega=0$, because $\partial_{x_{k}}\omega_{ij}=0,\; for all
\;i,j,k\in\left\{1,2,3,4\right\}$, with
$x_{1}=\eta_{1},\;x_{2}=\eta_{2},\;x_{3}=\dot{\theta}_{1},\;x_{4}=\dot{\theta}_{2}.$
\end{proof}
 Let $S$
be the lognormal statistical manifold.  Let $\left(TS,g\right)$ be
the K$\ddot{a}$hler manifold with $TS$ the bundle tangent to $S$.
Let $\mathbb{H}$ be the set of complex
 numbers such that\\ $i\mathbb{H}=\left\{z|\mathcal{Re}(z)<0\right\}$, and $\mathbb{H}=\left\{z|Im(z)>0\right\}$.
  In the coordinate system $\left(\theta_{1},\theta_{2},\dot{\theta}_{1},\dot{\theta}_{2}\right)$,
  the K$\ddot{a}$hler-Berndt metric is given by \begin{equation*} g_{KB}:=-\frac{1}{2\theta_{2}}
    \end{equation*}
 The Siegel Jacobi space $\mathbb{S}^{J}$ is defined by
\begin{equation*}
    \mathbb{S}^{J}:=\left(\mathbb{H}\times \mathbb{C},-\frac{1}{2\theta_{2}}\right)
\end{equation*}
$TS\cong \mathbb{S}^{J}$ and we have $TS\cong \mathbb{C}\times
i\mathbb{H}\longrightarrow\mathbb{H}\times
\mathbb{C},\;(z_{1},\;z_{2})\longmapsto (-iz_{2}, iz_{1})$
 with  $z_{1}=\theta_{1}+i\dot{\theta}_{1},\;z_{2}=\theta_{2}+i\dot{\theta}_{2}$.

\begin{proposition}\label{lem1}Let  $i,j,k\in\{1,2\}$, we have
the following result
\begin{eqnarray*}
 \left(\Gamma^{g}_{11} \right)^{1} &=& \left(\Gamma^{g}_{11} \right)^{2}
 =\left(\Gamma^{g}_{12} \right)^{1}=\left(\Gamma^{g}_{12} \right)^{2}=\left(\Gamma^{g}_{21} \right)^{1}
  =\left(\Gamma^{g}_{21} \right)^{2}=\left(\Gamma^{g}_{22} \right)^{1}=\left(\Gamma^{g}_{22}
  \right)^{2}=0.
\end{eqnarray*} with $\left(\Gamma^{g}_{ij} \right)^{k}$ is the
Christoffel's coefficient for the K$\ddot{a}$hler metric.
\end{proposition}
\begin{proof}
Using the proof  of theorem 4 in \cite{mama}, it is show that in the
lognormal statistical manifold the Christoffel coefficient is given
by
\begin{eqnarray}\Gamma_{ijk}=0,\; 1\leq i,j,k\leq 2.\label{eq18}\end{eqnarray} Using (\ref{eq18}) in lemma 2.22 in
\cite{dom}, we have the result.
\end{proof}

\section{K$\ddot{a}$hler function, and holomorphic isometric function.}\label{sec4}
\begin{theorem}\label{pro400}Let $\varphi: TS\longrightarrow TS$ be a diffeomorphism.
Every holomorphic isometric function $\varphi$ on $TS$ whose
elements satisfy $\varphi^{*}g=g$ and $\varphi_{*}J=J\varphi_{*}$ is
given by
\[\displaystyle {\it \varphi(\theta,\dot{\theta})}\, := \, \left( \theta_{{1}}+{\it k_{1}},\theta_{{2}}+{\it k_{2}},\dot{\theta}_{{1}}+{\it k_{3}},\dot{\theta}_{{2}}+{\it k_{4}}\\
\mbox{} \right) \] with $k_{1},\;k_{2},\;k_{3},\;k_{4}\in\mathrm{I\!
R}$.
\end{theorem}

\begin{proof}
Let
$\varphi=\left(\varphi^{1},\varphi^{2},\varphi^{3},\varphi^{4}\right)$,
we have the differential of $\varphi$ given by
\begin{eqnarray*}\varphi^{*}(\theta,\dot{\theta}) =\left(
                    \begin{array}{cccc}
                      \frac{\partial \varphi^{1}}{\partial \theta_{1}} & \frac{\partial \varphi^{1}}{\partial \theta_{2}}  & \frac{\partial \varphi^{1}}{\partial \dot{\theta}_{1}} & \frac{\partial \varphi^{1}}{\partial \dot{\theta}_{2}}  \\
                     \frac{\partial \varphi^{2}}{\partial \theta_{1}} & \frac{\partial \varphi^{2}}{\partial \theta_{2}}  & \frac{\partial \varphi^{2}}{\partial \dot{\theta}_{1}} & \frac{\partial \varphi^{2}}{\partial \dot{\theta}_{2}}  \\
                     \frac{\partial \varphi^{3}}{\partial \theta_{1}} & \frac{\partial \varphi^{3}}{\partial \theta_{2}}  & \frac{\partial \varphi^{3}}{\partial \dot{\theta}_{1}} & \frac{\partial \varphi^{3}}{\partial \dot{\theta}_{2}}  \\
                     \frac{\partial \varphi^{4}}{\partial \theta_{1}} & \frac{\partial \varphi^{4}}{\partial \theta_{2}}  & \frac{\partial \varphi^{4}}{\partial \dot{\theta}_{1}} & \frac{\partial \varphi^{4}}{\partial \dot{\theta}_{2}}  \\
                    \end{array}
                  \right)\end{eqnarray*}
                   Using remark \ref{remark10}, for that
                  $\varphi$ becomes holomorphic function we obtain
                  \begin{eqnarray*}
                        \frac{\partial \varphi^{1}}{\partial \theta_{1}} &=&  \frac{\partial \varphi^{3}}{\partial \dot{\theta}_{1}};\; \frac{ \partial \varphi^{1}}{\partial \theta_{2}} =  \frac{\partial \varphi^{3}}{\partial \dot{\theta}_{2}} \\
                        \frac{ \partial \varphi^{2}}{\partial \theta_{1}} &=& \frac{ \partial \varphi^{4}}{\partial \dot{\theta}_{1}};\; \frac{\partial \varphi^{2}}{\partial \theta_{2}} =  \frac{\partial \varphi^{4}}{\partial \dot{\theta}_{2}} \\
    \frac{\partial \varphi^{1}}{\partial \dot{\theta}_{1}} &=& - \frac{\partial \varphi^{3}}{\partial \theta_{1}};\;  \frac{\partial \varphi^{1}}{\partial \dot{\theta}_{2}} = - \frac{\partial \varphi^{3}}{\partial \theta_{2}} \\
    \frac{\partial \varphi^{2}}{\partial \dot{\theta}_{1}} &=& - \frac{\partial \varphi^{4}}{\partial \theta_{2}};\; \frac{\partial \varphi^{2}}{\partial \dot{\theta}_{1}} = - \frac{\partial \varphi^{4}}{\partial \theta_{2}}
                  \end{eqnarray*}
                  So, we obtain the following matrix  $\varphi^{*}(\theta,\dot{\theta}) =\left(
                    \begin{array}{cccc}
                      \frac{\partial \varphi^{1}}{\partial \theta_{1}} & \frac{\partial \varphi^{1}}{\partial \theta_{2}}  & \frac{\partial \varphi^{1}}{\partial \dot{\theta}_{1}} & \frac{\partial \varphi^{1}}{\partial \dot{\theta}_{2}}  \\
                     \frac{\partial \varphi^{2}}{\partial \theta_{1}} & \frac{\partial \varphi^{2}}{\partial \theta_{2}}  & \frac{\partial \varphi^{2}}{\partial \dot{\theta}_{1}} & \frac{\partial \varphi^{2}}{\partial \dot{\theta}_{2}}  \\
                     - \frac{\partial \varphi^{1}}{\partial \dot{\theta}_{1}}  & - \frac{\partial \varphi^{1}}{\partial \dot{\theta}_{2}}& \frac{\partial \varphi^{1}}{\partial \theta_{1}} & \frac{\partial \varphi^{1}}{\partial \theta_{2}}  \\
                     - \frac{\partial \varphi^{2}}{\partial \dot{\theta}_{1}} & - \frac{\partial \varphi^{2}}{\partial \dot{\theta}_{1}} & \frac{\partial \varphi^{2}}{\partial \theta_{1}} & \frac{\partial \varphi^{2}}{\partial \theta_{2}}   \\
                    \end{array}
                  \right)$\\
Using $\varphi^{*}g=g$  we have
\begin{equation*}
    \left(
                    \begin{array}{cccc}
                      \frac{\partial \varphi^{1}}{\partial \theta_{1}} & \frac{\partial \varphi^{1}}{\partial \theta_{2}}  & \frac{\partial \varphi^{1}}{\partial \dot{\theta}_{1}} & \frac{\partial \varphi^{1}}{\partial \dot{\theta}_{2}}  \\
                     \frac{\partial \varphi^{2}}{\partial \theta_{1}} & \frac{\partial \varphi^{2}}{\partial \theta_{2}}  & \frac{\partial \varphi^{2}}{\partial \dot{\theta}_{1}} & \frac{\partial \varphi^{2}}{\partial \dot{\theta}_{2}}  \\
                     - \frac{\partial \varphi^{1}}{\partial \dot{\theta}_{1}}  & - \frac{\partial \varphi^{1}}{\partial \dot{\theta}_{2}}& \frac{\partial \varphi^{1}}{\partial \theta_{1}} & \frac{\partial \varphi^{1}}{\partial \theta_{2}}  \\
                     - \frac{\partial \varphi^{2}}{\partial \dot{\theta}_{1}} & - \frac{\partial \varphi^{2}}{\partial \dot{\theta}_{1}} & \frac{\partial \varphi^{2}}{\partial \theta_{1}} & \frac{\partial \varphi^{2}}{\partial \theta_{2}}   \\
                    \end{array}
                  \right)\left(
        \begin{array}{cccc}
          -\frac{1}{2\theta_{2}} & \frac{\theta_{1}}{2\theta^{2}_{2}} & 0 & 0 \\
          \frac{\theta_{1}}{2\theta^{2}_{2}} & -\frac{\theta^{2}_{1}-\theta_{2}}{2\theta^{3}_{2}} & 0 & 0 \\
          0 & 0 &  -\frac{1}{2\theta_{2}} &  \frac{\theta_{1}}{2\theta^{2}_{2}} \\
0& 0 &  \frac{\theta_{1}}{2\theta^{2}_{2}} & -\frac{\theta^{2}_{1}-\theta_{2}}{2\theta^{3}_{2}} \\
        \end{array}
    \right)=\left(
        \begin{array}{cccc}
          -\frac{1}{2\theta_{2}} & \frac{\theta_{1}}{2\theta^{2}_{2}} & 0 & 0 \\
          \frac{\theta_{1}}{2\theta^{2}_{2}} & -\frac{\theta^{2}_{1}-\theta_{2}}{2\theta^{3}_{2}} & 0 & 0 \\
          0 & 0 &  -\frac{1}{2\theta_{2}} &  \frac{\theta_{1}}{2\theta^{2}_{2}} \\
0& 0 &  \frac{\theta_{1}}{2\theta^{2}_{2}} & -\frac{\theta^{2}_{1}-\theta_{2}}{2\theta^{3}_{2}} \\
        \end{array}
    \right)
\end{equation*}
Using  $\varphi_{*}J=J\varphi_{*}$, we obtain the following system
\begin{equation*}
    \left\{
      \begin{array}{ll}
        -\frac{1}{2\theta_{2}}\frac{\partial \varphi^{1}}{\partial \theta_{1}} + \frac{\theta_{1}}{2\theta^{2}_{2}}\frac{\partial \varphi^{1}}{\partial \theta_{2}}=-\frac{1}{2\theta_{2}}& \hbox{} \\
 \frac{\theta_{1}}{2\theta^{2}_{2}}\frac{\partial \varphi^{1}}{\partial \theta_{1}}   -\frac{\theta^{2}_{1}-\theta_{2}}{2\theta^{3}_{2}}\frac{\partial \varphi^{1}}{\partial \theta_{2}}=\frac{\theta_{1}}{2\theta^{2}_{2}} & \hbox{} \\
-\frac{1}{2\theta_{2}}\frac{\partial \varphi^{1}}{\partial \dot{\theta}_{1}}  + \frac{\theta_{1}}{2\theta^{2}_{2}}\frac{\partial \varphi^{1}}{\partial \dot{\theta}_{2}}=0& \hbox{} \\
\frac{\theta_{1}}{2\theta^{2}_{2}}\frac{\partial \varphi^{1}}{\partial \dot{\theta}_{1}}   -\frac{\theta^{2}_{1}-\theta_{2}}{2\theta^{3}_{2}}\frac{\partial \varphi^{1}}{\partial \dot{\theta}_{2}}=0& \hbox{} \\
        -\frac{1}{2\theta_{2}}\frac{\partial \varphi^{2}}{\partial \theta_{1}} + \frac{\theta_{1}}{2\theta^{2}_{2}}\frac{\partial \varphi^{2}}{\partial \theta_{2}}=\frac{\theta_{1}}{2\theta^{2}_{2}}& \hbox{} \\
 \frac{\theta_{1}}{2\theta^{2}_{2}}\frac{\partial \varphi^{2}}{\partial \theta_{1}}   -\frac{\theta^{2}_{1}-\theta_{2}}{2\theta^{3}_{2}}\frac{\partial \varphi^{2}}{\partial \theta_{2}}=-\frac{\theta^{2}_{1}-\theta_{2}}{2\theta^{3}_{2}} & \hbox{} \\
-\frac{1}{2\theta_{2}}\frac{\partial \varphi^{2}}{\partial \dot{\theta}_{1}}  + \frac{\theta_{1}}{2\theta^{2}_{2}}\frac{\partial \varphi^{2}}{\partial \dot{\theta}_{2}}=0& \hbox{} \\
\frac{\theta_{1}}{2\theta^{2}_{2}}\frac{\partial \varphi^{2}}{\partial \dot{\theta}_{1}}   -\frac{\theta^{2}_{1}-\theta_{2}}{2\theta^{3}_{2}}\frac{\partial \varphi^{2}}{\partial \dot{\theta}_{2}}=0& \hbox{} \\
 \frac{1}{2\theta_{2}}\frac{\partial \varphi^{1}}{\partial \dot{\theta}_{1}} - \frac{\theta_{1}}{2\theta^{2}_{2}}\frac{\partial \varphi^{1}}{\partial \dot{\theta}_{2}}=0& \hbox{} \\
 -\frac{\theta_{1}}{2\theta^{2}_{2}}\frac{\partial \varphi^{1}}{\partial \dot{\theta}_{1}}   +\frac{\theta^{2}_{1}-\theta_{2}}{2\theta^{3}_{2}}\frac{\partial \varphi^{1}}{\partial \dot{\theta}_{2}}=0 & \hbox{} \\
-\frac{1}{2\theta_{2}}\frac{\partial \varphi^{1}}{\partial \theta_{1}}  + \frac{\theta_{1}}{2\theta^{2}_{2}}\frac{\partial \varphi^{1}}{\partial \theta_{2}}=-\frac{1}{2\theta_{2}}& \hbox{} \\
\frac{\theta_{1}}{2\theta^{2}_{2}}\frac{\partial \varphi^{1}}{\partial \theta_{1}}   -\frac{\theta^{2}_{1}-\theta_{2}}{2\theta^{3}_{2}}\frac{\partial \varphi^{1}}{\partial \theta_{2}}=\frac{\theta_{1}}{2\theta^{2}_{2}}& \hbox{} \\
\frac{1}{2\theta_{2}}\frac{\partial \varphi^{2}}{\partial \dot{\theta}_{1}} - \frac{\theta_{1}}{2\theta^{2}_{2}}\frac{\partial \varphi^{2}}{\partial \dot{\theta}_{2}}=0& \hbox{} \\
-\frac{\theta_{1}}{2\theta^{2}_{2}}\frac{\partial \varphi^{2}}{\partial \dot{\theta}_{1}}   +\frac{\theta^{2}_{1}-\theta_{2}}{2\theta^{3}_{2}}\frac{\partial \varphi^{2}}{\partial \dot{\theta}_{2}}=0 & \hbox{} \\
-\frac{1}{2\theta_{2}}\frac{\partial \varphi^{2}}{\partial \theta_{1}}  + \frac{\theta_{1}}{2\theta^{2}_{2}}\frac{\partial \varphi^{2}}{\partial \theta_{2}}=\frac{\theta_{1}}{2\theta^{2}_{2}}& \hbox{} \\
\frac{\theta_{1}}{2\theta^{2}_{2}}\frac{\partial
\varphi^{2}}{\partial \theta_{1}}
-\frac{\theta^{2}_{1}-\theta_{2}}{2\theta^{3}_{2}}\frac{\partial
\varphi^{2}}{\partial
\theta_{2}}=-\frac{\theta^{2}_{1}-\theta_{2}}{2\theta^{3}_{2}}&
\hbox{}
      \end{array}
    \right.
\end{equation*}

we have, \begin{eqnarray*}
  \varphi^{1}(\theta,\dot{\theta})&=&\theta_{1}+{\it k_{1}},\varphi^{2}(\theta,\dot{\theta})=
  \theta_{2}+{\it k_{2}},\varphi^{3}(\theta,\dot{\theta})=\dot{\theta}_{1}+{\it k_{3}},\varphi^{4}(\theta,\dot{\theta})=\dot{\theta}_{2}+{\it k_{4}}
\end{eqnarray*}
So, we have \[\displaystyle {\it \varphi(\theta,\dot{\theta})}\, :=
\, \left( \theta_{{1}}+{\it k_{1}},\theta_{{2}}+{\it
k_{2}},\dot{\theta}_{{1}}
+{\it k_{3}},\dot{\theta}_{{2}}+{\it k_{4}}\\
\mbox{} \right) \] the homomorphic function.  With
$k_{1},k_{2},k_{3},k_{4}\in \mathrm{I\! R}$. Therefore, this
function $\varphi$ verify lemma \ref{lem4}. So,
$\varphi(\theta,\dot{\theta})$ is holomorphic isometric function.
\end{proof}

\begin{proposition}\label{pro40}
Let  $L\in \left\{F,G,H,H,P,Q\in\mathfrak{g}^{J}\right\}$ the
Hamiltonian vector field $X_{J}L$ on Siegel-Jacobi space is given by
\begin{eqnarray*}
  X_{\psi(F)}(\theta,\dot{\theta})&=&\left(0,0,0,0\right),\;\;\;\;\; X_{\psi(P)}(\theta,\dot{\theta})=\left(2\theta_{1}^{2}-2\theta_{2},2\theta_{1}\theta_{2},-2\theta_{1}\theta_{2},-2\theta_{2}^{2}\right)\\
  X_{\psi(G)}(\theta,\dot{\theta})&=&\left(2\theta_{1}^{2}-2\theta_{2} +2\theta_{1}\theta_{2}\dot{\theta}_{2},2\theta_{1}\theta_{2}+2\theta_{2}^{2}\dot{\theta}_{2},-2\theta_{1}\theta_{2}(\theta_{2}+1),-2\theta_{2}^{2}(\theta_{2}+1)\right),\\
  X_{\psi(H)}(\theta,\dot{\theta})&=&\left(0,0,-2\theta_{1}^{2}+2\theta_{2} ,-2\theta_{1}\theta_{2}\right),\;\;\;\;\;X_{\psi(R)}(\theta,\dot{\theta})=\left(0,0,0,0\right)\\
  X_{\psi(Q)}(\theta,\dot{\theta})&=&\left(0,0,-2\theta_{1}\theta_{2},-2\theta_{2}^{2} \right)
\end{eqnarray*}
\end{proposition}
\begin{proof}
In the coordinate system
$\left(\theta_{1},\theta_{2},\dot{\theta}_{1},\dot{\theta}_{2}\right)$,
using equation (\ref{eqq})  we have the following system
\begin{equation}\label{syst3}
\left\{
  \begin{array}{ll}
    \frac{\partial^{2} f}{\partial\theta_{1}\partial\theta_{1}} -\frac{\partial^{2} f}{\partial\dot{\theta}_{1}\partial\dot{\theta}_{1}}
    =2\left(\Gamma^{h}_{11} \right)^{1}\circ \pi\frac{\partial f}{\partial\theta_{1}}+2\left(\Gamma^{h}_{11} \right)^{2}\circ \pi\frac{\partial f}{\partial\theta_{1}}& \hbox{} \\
     \frac{\partial^{2} f}{\partial\theta_{1}\partial\theta_{2}} -\frac{\partial^{2} f}{\partial\dot{\theta}_{1}\partial\dot{\theta}_{2}}
    =2\left(\Gamma^{h}_{12} \right)^{1}\circ \pi\frac{\partial f}{\partial\theta_{1}}+2\left(\Gamma^{h}_{12} \right)^{2}\circ \pi\frac{\partial f}{\partial\theta_{2}} & \hbox{} \\
    \frac{\partial^{2} f}{\partial\theta_{2}\partial\theta_{1}} -\frac{\partial^{2} f}{\partial\dot{\theta}_{2}\partial\dot{\theta}_{1}}
    =2\left(\Gamma^{h}_{21} \right)^{1}\circ \pi\frac{\partial f}{\partial\theta_{1}}+2\left(\Gamma^{h}_{21} \right)^{2}\circ \pi\frac{\partial f}{\partial\theta_{2}} & \hbox{} \\
    \frac{\partial^{2} f}{\partial\theta_{2}\partial\theta_{2}} -\frac{\partial^{2} f}{\partial\dot{\theta}_{2}\partial\dot{\theta}_{2}}
    =2\left(\Gamma^{h}_{22} \right)^{1}\circ \pi\frac{\partial f}{\partial\theta_{1}}+2\left(\Gamma^{h}_{22} \right)^{2}\circ \pi\frac{\partial f}{\partial\theta_{2}} & \hbox{} \\
    \frac{\partial^{2} f}{\partial\theta_{1}\partial\dot{\theta}_{1}} +\frac{\partial^{2} f}{\partial\theta_{1}\partial\dot{\theta}_{1}}
    =2\left(\Gamma^{h}_{11} \right)^{1}\circ \pi\frac{\partial f}{\partial\dot{\theta}_{1}}+2\left(\Gamma^{h}_{11} \right)^{2}\circ \pi\frac{\partial f}{\partial\dot{\theta}_{2}} & \hbox{} \\
    \frac{\partial^{2} f}{\partial\theta_{1}\partial\dot{\theta}_{2}} +\frac{\partial^{2} f}{\partial\theta_{2}\partial\dot{\theta}_{1}}
    =2\left(\Gamma^{h}_{12} \right)^{1}\circ \pi\frac{\partial f}{\partial\dot{\theta}_{1}}+2\left(\Gamma^{h}_{12} \right)^{2}\circ \pi\frac{\partial f}{\partial\dot{\theta}_{2}} & \hbox{} \\
    \frac{\partial^{2} f}{\partial\theta_{2}\partial\dot{\theta}_{1}} +\frac{\partial^{2} f}{\partial\theta_{1}\partial\dot{\theta}_{2}}
    =2\left(\Gamma^{h}_{21} \right)^{1}\circ \pi\frac{\partial f}{\partial\dot{\theta}_{1}}+2\left(\Gamma^{h}_{21} \right)^{2}\circ \pi\frac{\partial f}{\partial\dot{\theta}_{2}} & \hbox{} \\
    \frac{\partial^{2} f}{\partial\theta_{2}\partial\dot{\theta}_{2}} +\frac{\partial^{2} f}{\partial\theta_{2}\partial\dot{\theta}_{2}}
    =2\left(\Gamma^{h}_{22} \right)^{1}\circ \pi\frac{\partial f}{\partial\dot{\theta}_{1}}+2\left(\Gamma^{h}_{22} \right)^{2}\circ \pi\frac{\partial f}{\partial\dot{\theta}_{2}} & \hbox{}
  \end{array}
\right.
\end{equation}

Using the lemma 2.22 in \cite{dom}, we have the following system

\begin{equation}\label{eq19}
\left\{
  \begin{array}{ll}
    \frac{\partial^{2} f}{\partial\theta_{1}\partial\theta_{1}} -\frac{\partial^{2} f}{\partial\dot{\theta}_{1}\partial\dot{\theta}_{1}}
    =2\left(\Gamma^{g}_{11} \right)^{1}\frac{\partial f}{\partial\theta_{1}}+2\left(\Gamma^{g}_{11} \right)^{2}\frac{\partial f}{\partial\theta_{1}}& \hbox{} \\
     \frac{\partial^{2} f}{\partial\theta_{1}\partial\theta_{2}} -\frac{\partial^{2} f}{\partial\dot{\theta}_{1}\partial\dot{\theta}_{2}}
    =2\left(\Gamma^{g}_{12} \right)^{1}\frac{\partial f}{\partial\theta_{1}}+2\left(\Gamma^{g}_{12} \right)^{2}\frac{\partial f}{\partial\theta_{2}} & \hbox{} \\
    \frac{\partial^{2} f}{\partial\theta_{2}\partial\theta_{1}} -\frac{\partial^{2} f}{\partial\dot{\theta}_{2}\partial\dot{\theta}_{1}}
    =2\left(\Gamma^{g}_{21} \right)^{1}\frac{\partial f}{\partial\theta_{1}}+2\left(\Gamma^{g}_{21} \right)^{2}\frac{\partial f}{\partial\theta_{2}} & \hbox{} \\
    \frac{\partial^{2} f}{\partial\theta_{2}\partial\theta_{2}} -\frac{\partial^{2} f}{\partial\dot{\theta}_{2}\partial\dot{\theta}_{2}}
    =2\left(\Gamma^{g}_{22} \right)^{1}\frac{\partial f}{\partial\theta_{1}}+2\left(\Gamma^{g}_{22} \right)^{2}\frac{\partial f}{\partial\theta_{2}} & \hbox{} \\
    \frac{\partial^{2} f}{\partial\theta_{1}\partial\dot{\theta}_{1}} +\frac{\partial^{2} f}{\partial\theta_{1}\partial\dot{\theta}_{1}}
    =2\left(\Gamma^{g}_{11} \right)^{1}\frac{\partial f}{\partial\dot{\theta}_{1}}+2\left(\Gamma^{g}_{11} \right)^{2}\frac{\partial f}{\partial\dot{\theta}_{2}} & \hbox{} \\
    \frac{\partial^{2} f}{\partial\theta_{1}\partial\dot{\theta}_{2}} +\frac{\partial^{2} f}{\partial\theta_{2}\partial\dot{\theta}_{1}}
    =2\left(\Gamma^{g}_{12} \right)^{1}\frac{\partial f}{\partial\dot{\theta}_{1}}+2\left(\Gamma^{g}_{12} \right)^{2}\frac{\partial f}{\partial\dot{\theta}_{2}} & \hbox{} \\
    \frac{\partial^{2} f}{\partial\theta_{2}\partial\dot{\theta}_{1}} +\frac{\partial^{2} f}{\partial\theta_{1}\partial\dot{\theta}_{2}}
    =2\left(\Gamma^{g}_{21} \right)^{1}\frac{\partial f}{\partial\dot{\theta}_{1}}+2\left(\Gamma^{g}_{21} \right)^{2}\frac{\partial f}{\partial\dot{\theta}_{2}} & \hbox{} \\
    \frac{\partial^{2} f}{\partial\theta_{2}\partial\dot{\theta}_{2}} +\frac{\partial^{2} f}{\partial\theta_{2}\partial\dot{\theta}_{2}}
    =2\left(\Gamma^{g}_{22} \right)^{1}\frac{\partial f}{\partial\dot{\theta}_{1}}+2\left(\Gamma^{g}_{22} \right)^{2}\frac{\partial f}{\partial\dot{\theta}_{2}} & \hbox{}
  \end{array}
\right.\end{equation} using lemma \ref{lem1}, we have the following
system
\begin{equation}\label{eq20}
\left\{
  \begin{array}{ll}
    \frac{\partial^{2} f}{\partial\theta_{1}\partial\theta_{1}} -\frac{\partial^{2} f}{\partial\dot{\theta}_{1}\partial\dot{\theta}_{1}}
    =0& \hbox{} \\
     \frac{\partial^{2} f}{\partial\theta_{1}\partial\theta_{2}} -\frac{\partial^{2} f}{\partial\dot{\theta}_{1}\partial\dot{\theta}_{2}}
    =0 & \hbox{} \\
    \frac{\partial^{2} f}{\partial\theta_{2}\partial\theta_{1}} -\frac{\partial^{2} f}{\partial\dot{\theta}_{2}\partial\dot{\theta}_{1}}
    =0 & \hbox{} \\
    \frac{\partial^{2} f}{\partial\theta_{2}\partial\theta_{2}} -\frac{\partial^{2} f}{\partial\dot{\theta}_{2}\partial\dot{\theta}_{2}}
    =0 & \hbox{} \\
    \frac{\partial^{2} f}{\partial\theta_{1}\partial\dot{\theta}_{1}} +\frac{\partial^{2} f}{\partial\theta_{1}\partial\dot{\theta}_{1}}
    =0 & \hbox{} \\
    \frac{\partial^{2} f}{\partial\theta_{1}\partial\dot{\theta}_{2}} +\frac{\partial^{2} f}{\partial\theta_{2}\partial\dot{\theta}_{1}}
    =0 & \hbox{} \\
    \frac{\partial^{2} f}{\partial\theta_{2}\partial\dot{\theta}_{1}} +\frac{\partial^{2} f}{\partial\theta_{1}\partial\dot{\theta}_{2}}
    =0 & \hbox{} \\
    \frac{\partial^{2} f}{\partial\theta_{2}\partial\dot{\theta}_{2}} +\frac{\partial^{2} f}{\partial\theta_{2}\partial\dot{\theta}_{2}}
    =0 & \hbox{}
  \end{array}
\right.\end{equation}

We obtain the following family K$\ddot{a}$hler function define by
\[\displaystyle
    f(\theta,\dot{\theta})= \frac{1}{2}\,{\it \alpha_{1}}\,{\theta_{{1}}}^{2}
    +\frac{1}{2}\, \left( 2\,{\it \alpha_{2}}\,\theta_{{2}}-2\,{\it \alpha_{4}}\,\theta_{{4}}+2\,{\it \alpha_{5}} \right) \theta_{{1}}
    +\frac{1}{2}\,{\it \alpha_{9}}\,{\theta_{{2}}}^{2}\\
\mbox{}+\frac{1}{2}\, \left( 2\,\theta_{{3}}{\it \alpha_{4}}+2\,{\it
\alpha_{6}} \right) \theta_{{2}}\]
\[\displaystyle+\frac{1}{2}\,{\it\alpha_{1}}\,{\theta_{{3}}}^{2}+\frac{1}{2}\, \left(
2\,{\it \alpha_{3}}+2\,{\it \alpha_{2}}\,\theta_{{4}} \right)
\theta_{{3}}
\] with
$\alpha_{1},\;\alpha_{2},\;\alpha_{3},\;\alpha_{4},\;\alpha_{5},\;\alpha_{6}\in\mathrm{I\! R}$.\\

Having this in mind, let $C^{\infty}(\mathbb{S}^{J} )$ denote the
space of smooth functions on the Siegel-Jacobi space
$\mathbb{S}^{J}$. Using the  coordinates
$(\theta_{1},\theta_{2},\dot{\theta}_{1},\dot{\theta}_{2})$ on
$\mathbb{S}^{J}\cong TS$, we define a linear map $\psi:
\mathfrak{g}^{J}\longrightarrow C^{\infty}(\mathbb{S}^{J} )$ as
follows
\begin{eqnarray*}
  F&\longmapsto&0,\;\;\;\;\;P\longmapsto\theta_{2}+\dot{\theta}_{1}\\
  G&\longmapsto&\frac{1}{2}\theta^{2}_{2}+\theta_{2}+\dot{\theta}_{1}+\frac{1}{2}\dot{\theta}^{2}_{2},\;\;\;\;\;Q\longmapsto \theta_{2}\\
  H&\longmapsto&\theta_{1},\;\;\;\;\;R\longmapsto-\frac{1}{4}.
\end{eqnarray*}
Using the formula $X_{f}(\theta,\dot{\theta})=h^{ij}\frac{\partial
f}{\partial\dot{\theta}_{i}}\frac{\partial
}{\partial\theta_{j}}-h^{ij}\frac{\partial f
}{\partial\theta_{i}}\frac{\partial }{\partial\dot{\theta}_{j}}$,
and using (\ref{eq12}) the Hamiltonian vector fields $X_{J}(L)$ is
given by
\begin{eqnarray*}
  X_{\psi(F)}(\theta,\dot{\theta})&=&\left(0,0,0,0\right),\;\;\;\;\; X_{\psi(P)}(\theta,\dot{\theta})=\left(2\theta_{1}^{2}-2\theta_{2},2\theta_{1}\theta_{2},-2\theta_{1}\theta_{2},-2\theta_{2}^{2}\right)\\
  X_{\psi(G)}(\theta,\dot{\theta})&=&\left(2\theta_{1}^{2}-2\theta_{2} +2\theta_{1}\theta_{2}\dot{\theta}_{2},2\theta_{1}\theta_{2}+2\theta_{2}^{2}\dot{\theta}_{2},-2\theta_{1}\theta_{2}(\theta_{2}+1),-2\theta_{2}^{2}(\theta_{2}+1)\right),\\
  X_{\psi(H)}(\theta,\dot{\theta})&=&\left(0,0,-2\theta_{1}^{2}+2\theta_{2} ,-2\theta_{1}\theta_{2}\right),\;\;\;\;\;X_{\psi(R)}(\theta,\dot{\theta})=\left(0,0,0,0\right)\\
  X_{\psi(Q)}(\theta,\dot{\theta})&=&\left(0,0,-2\theta_{1}\theta_{2},-2\theta_{2}^{2} \right)
\end{eqnarray*}

\end{proof}

\section{Shr$\ddot{o}$dinger equation on Jacobi space.}\label{sec5}
Let $\mathbb{S}^{J}$ the Siegel-Jacobi Space. Let $
\mathcal{H}:=L^{2}(\mathrm{I\! R})$  be the Hilbert space of
square-integrable functions.  By setting $t=(z_{1},\;z_{2}), \;
I=\mathbb{C}\times i\mathbb{H}$, We define the exponential
application
\begin{eqnarray*}
  \Psi:  \mathbb{S}^{J}&\longrightarrow& L^{2}(\mathrm{I\! R}) \\
  \alpha(t)&\mapsto& \exp\left\{\frac{1}{2}\left(c(x)-iz_{2}\log(x)+iz_{1}\log^{2}(x)-\Phi(\theta)\right)\right\}
\end{eqnarray*}
with $c(x)=-\log(x)-i\log(x)$,\; $\Phi(\theta)=-
\frac{\theta^{2}_{1}}{4\theta_{2}}-\frac{1}{2}\log(-2\theta_{2})+\frac{1}{2}\log(2\pi)$.
We have   $\Phi(t):= \Psi\left(\alpha(t)\right)$.  We have
$\alpha:I\longrightarrow \mathbb{H}\times
\mathbb{C}=\mathbb{S}^{J},\; \alpha(t)\longmapsto(-iz_{2},\;z_{1})$,
we have the following theorem

\begin{theorem}\label{pro4} Let $\mathbb{S}^{J}$ be a Jacobian space and $\mathfrak{g}^{J}$ be the Lie algebra
 of the Jacobian group $G^{J}(\mathrm{I\! R})$.  $\mathfrak{s}L(2,\mathrm{I\! R})$ and $\mathfrak{h}$ are
 respectively the Lie algebra of the special linear group $SL(2,\mathrm{I\! R})$ of dimension $3$ and the Heisenberg group $Heis$.
 For every smooth map $\kappa:\mathbb{S}^{J}\times \mathfrak{g}^{J}\longrightarrow\mathbb{C}$
 there exists $\alpha:I \longrightarrow \mathbb{H}\times \mathbb{C}=\mathbb{S}^{J},\; \alpha(t)\longmapsto(-iz_{2},\;z_{1})$
 an integral curve of the Hamiltonian vector field $X_{J}L$, and constant application of $K_{L}\Phi(t):=\kappa\left(\alpha(t),L\right)=cste$,
 with $t=(z_{1},\;z_{2}), \; I=\mathbb{C}\times i\mathbb{H}$, and satisfying the following equation on the Jacobi space $\mathbb{S}^{J}$  is given by
\begin{eqnarray}\label{eq21}
    \frac{id
\Phi(t)}{dt}&=&\frac{1}{2}\left(-\lambda_{1}\log^{2}(x)-\lambda_{2}\left(2iz_{1}+\left(-iz_{2}+2iz_{1}\log(x)\right)^{2}\right)\right)\Phi(t)\nonumber\\
&+&\frac{1}{2}\left(\lambda_{3}\left(2\log(x)z_{2}-4z_{1}\log^{2}(x)+i\right)+\lambda_{4}\left(-z_{2}+2z_{1}\log(x)\right)\Phi(t)\right)\\
&+&\frac{1}{2}\left(\xi(\theta,\dot{\theta},x)\log(x)-\frac{1}{4}\lambda_{5}\right)\Phi(t)+\frac{1}{2}\left(\lambda_{1}K_{F}\Phi(t)
+\lambda_{2}K_{G}\Phi(t)+\lambda_{3}K_{H}\Phi(t)\right.\nonumber\\
&+&\left.\lambda_{4}K_{P}\Phi(t)+\xi(\theta,\dot{\theta},x)K_{Q}\Phi(t)+\lambda_{5}K_{R}\Phi(t)\right)\nonumber
\end{eqnarray}
with \begin{eqnarray*}
    \xi(\theta,\dot{\theta},x)&=&
    \frac{\left(-4i\ddot{\theta}_{1}\log^{2}(x)
    -4i\log(x)\ddot{\theta}_{2}
     +4\lambda_{1}\log^{2}(x)
       +8i\lambda_{2}\dot{\theta}_{1}\theta_{2}+\lambda_{5}-4i\lambda_{3}\right)\varphi(t)}{4\left(\beta_{5}+\log(x)\right)\varphi(t)}\\
 &+&\frac{\left(+16\lambda_{3}\theta_{1}\log^{2}(x)
 -8\lambda_{4}\theta_{1}\log(x)
 +8\lambda_{2}\theta_{1}\theta_{2}
 +8i\lambda_{2}\theta_{1}+4i\lambda_{4}\dot{\theta}_{2}
-8\lambda_{2}\dot{\theta}_{2}\dot{\theta}_{1}\right)\Phi(t)}{4\left(\beta_{5}+\log(x)\right)\Phi(t)}\\
 &&+\frac{\left(2\lambda_{2}\dot{\theta}^{2}_{2}
 +8\lambda_{2}\dot{\theta}_{1}+4\lambda_{4}\theta_{2}
 -8\lambda_{2}\dot{\theta}_{1}-2\lambda_{2}\theta^{2}_{2}-8\lambda_{2}\theta^{2}_{1}
 -4\dot{\theta}_{1}\log^{2}(x)-4\dot{\theta}_{2}\log(x)\right)\Phi(t)}{4\left(\beta_{5}+\log(x)\right)\Phi(t)}\\
 &&+\frac{
 \left(-8\lambda_{3}\left(\theta_{2}+i\dot{\theta}_{2}\right)\log(x)
 +8\lambda_{2}\theta_{1}\dot{\theta}_{2}
 -16i\lambda_{2}\theta_{1}\dot{\theta}_{1}
 +16i\lambda_{3}\dot{\theta}_{1}\log^{2}(x)
 \right.}{4\left(\beta_{5}+\log(x)\right)\Phi(t)}\\
&&+\frac{
 \left.
 -8i\lambda_{4}\dot{\theta}_{2}
 -4i\lambda_{2}\theta_{2}\dot{\theta}_{2}\right)\Phi(t)}{4\left(\beta_{5}+\log(x)\right)\Phi(t)}\\
&+&\frac{-4\lambda_{1}\beta_{1}-4\lambda_{2}\beta_{2}
-4\lambda_{2}\beta_{2}
-4\lambda_{3}\beta_{3}-4\lambda_{4}\beta_{4}-4\lambda_{5}\beta_{6}
}{4\left(\beta_{5}+\log(x)\right)\Phi(t)}.
\end{eqnarray*} Where,
 $\lambda_{1}, \;\lambda_{2}, \; \lambda_{3}, \; \lambda_{4},
\;\lambda_{5},\;\beta_{1},\;\beta_{2},\;\beta_{3},\;\beta_{4},\;\beta_{5},\;\beta_{6}\in\mathrm{I\!R}
$.
\end{theorem}

\begin{proof}
 $End\left(\mathcal{C}^{\infty}(\mathrm{I\! R},\mathbb{C}
)\right)$ denotes the space of $\mathbb{C}$-linear endomorphisms of
$\mathcal{C}^{\infty}(\mathrm{I\! R},\mathbb{C} )$, and let
\begin{eqnarray}\label{eq22}
  Q:\mathfrak{g}^{J}&\longrightarrow& End\left(C^{\infty}(\mathrm{I\! R},\mathbb{C}
  )\right)\nonumber\\
  F&\longmapsto&-\log^{2}(x),\;\;\;\;\;P\longmapsto-i\frac{\partial}{\partial \log(x)}\\
  G&\longmapsto&-\frac{\partial}{\partial \log^{2}(x)},\;\;\;\;\;Q\longmapsto \log(x)\nonumber\\\\
  H&\longmapsto&2i\left(\log(x) \frac{\partial}{\partial
  \log(x)}+\frac{1}{2}I\right),\;\;\;\;\;R\longmapsto-\frac{1}{4}I.\nonumber\\
\end{eqnarray}
 be the linear map.
Let $\mathbb{S}^{J}$ the Siegel-Jacobi Space. Let $
\mathcal{H}:=L^{2}(\mathrm{I\! R})$  be the Hilbert space of
square-integrable functions.  By setting $t=(z_{1},\;z_{2}), \;
I=\mathbb{C}\times i\mathbb{H}$, We define the exponential
application
\begin{eqnarray*}
  \Psi:  \mathbb{S}^{J}&\longrightarrow& L^{2}(\mathrm{I\! R}) \\
  \alpha(t)&\mapsto& \exp\left\{\frac{1}{2}\left(c(x)-iz_{2}\log(x)+iz_{1}\log^{2}(x)-\Phi(\theta)\right)\right\}
\end{eqnarray*}
with $c(x)=-\log(x)-i\log(x)$,\; $\Phi(\theta)=-
\frac{\theta^{2}_{1}}{4\theta_{2}}-\frac{1}{2}\log(-2\theta_{2})+\frac{1}{2}\log(2\pi)$.
We have   $\Phi(t):= \Psi\left(\alpha(t)\right)$.  We have
$\alpha:I\longrightarrow \mathbb{H}\times
\mathbb{C}=\mathbb{S}^{J},\; \alpha(t)\longmapsto(-iz_{2},\;z_{1})$,
we find the application $K_{L}$ such that the following
Shr$\ddot{o}$dinger equation is verify. By definition $\Phi(t):=
\Psi\left(\alpha(t)\right)$, and we have

\begin{eqnarray*}
\frac{d  \Phi(t)}{dt}&=&\frac{d  \Psi\left(\alpha(t)\right)}{dt}
\end{eqnarray*}

So,\begin{eqnarray*} \frac{id
\Psi\left(\alpha(t)\right)}{dt}&=&i\frac{\partial
\Psi\left(\alpha(t)\right)}{\partial (iz_{2})}\frac{d  (iz_{2})}{dt}
+i\frac{\partial \Psi\left(\alpha(t)\right)}{\partial (z_{1})}\frac{d  (z_{1})}{dt}\\
&=&i\frac{\partial \Psi\left(\alpha(t)\right)}{\partial
(iz_{2})}\left(\frac{\partial (iz_{2})}{\partial \theta_{1}}\frac{d
\theta_{1}}{dt}+\frac{\partial (iz_{2})}{\partial \theta_{2}}\frac{d
\theta_{2}}{dt}+\frac{\partial (iz_{2})}{\partial
\dot{\theta}_{1}}\frac{d \dot{\theta}_{1}}{dt}+\frac{\partial
(iz_{2})}{\partial \dot{\theta}_{2}}\frac{d
\dot{\theta}_{2}}{dt}\right)\\
 &+&i\frac{\partial \Psi\left(\alpha(t)\right)}{\partial
(z_{1})}\left(\frac{\partial (z_{1})}{\partial \theta_{1}}\frac{d
\theta_{1}}{dt}+\frac{\partial (z_{1})}{\partial \theta_{2}}\frac{d
\theta_{2}}{dt}+\frac{\partial (z_{1})}{\partial
\dot{\theta}_{1}}\frac{d \dot{\theta}_{1}}{dt}+\frac{\partial
(z_{1})}{\partial \dot{\theta}_{2}}\frac{d
\dot{\theta}_{2}}{dt}\right)\end{eqnarray*}So, we obtain the
following expression
\begin{eqnarray}\label{eq23}
 \frac{id
\Psi\left(\alpha(t)\right)}{dt}
&=&\frac{1}{2}\left(-\log^{2}(x)(\dot{\theta}_{1}+i\ddot{\theta}_{1})+\log(x)(-\dot{\theta}_{2}-i\ddot{\theta}_{2})\right)\Psi
(\alpha(t))
\end{eqnarray}

we have
\begin{eqnarray*}
  Q(F)\Phi(t) &=& -\log^{2}(x)\Psi (\alpha(t)),\;\;\;\;\; Q(P)\Phi(t)= \left(-z_{2}+2z_{1}\log(x)\right)\Psi(\alpha(t)) \\
   Q(G)\Phi(t) &=& -\left(2iz_{1}+\left(-iz_{2}+2iz_{1}\log(x)\right)^{2}\right)\Psi  (\alpha(t)),\;\;\;\;\; Q(Q)\Phi(t)= \log(x)\Psi(\alpha(t)) \\
   Q(H)\Phi(t) &=& \left(2\log(x)z_{2}-4z_{1}\log^{2}(x)+i\right)\Psi  (\alpha(t)),\;\;\;\;\; Q(R)\Phi(t) = -\frac{1}{4}\Psi(\alpha(t))
\end{eqnarray*}

Taking linear combination of the following Hermitian operators
acting on $\mathcal{C}^{\infty}(\mathrm{I\! R},\mathbb{C} )$. So we
have
\begin{eqnarray*}
    \frac{id
\Psi\left(\alpha(t)\right)}{dt}&=&A\left(-\frac{1}{2}\log^{2}(x)\Psi\left(\alpha(t)\right)+\frac{1}{2}
K_{F}\Psi\left(\alpha(t)\right)\right)\\
&+&B\left(-\frac{1}{2}\left(2iz_{1}+\left(-iz_{2}+2iz_{1}\log(x)\right)^{2}\right)\Psi(\alpha(t))+\frac{1}{2}K_{G}\Psi\left(\alpha(t)\right)\right)\\
&+&C\left(\frac{1}{2}\left(2\log(x)z_{2}-4z_{1}\log^{2}(x)+i\right)\Psi
(\alpha(t))+\frac{1}{2}K_{H}\Psi\left(\alpha(t)\right)\right)\\
&+&D\left(\frac{1}{2}\left(-z_{2}+2z_{1}\log(x)\right)\Psi(\alpha(t))+\frac{1}{2}K_{P}\Psi\left(\alpha(t)\right)\right)\\
&+&\xi\left(\frac{1}{2}\log(x)\Psi(\alpha(t))+\frac{1}{2}K_{Q}\Psi\left(\alpha(t)\right)\right)+F\left(-\frac{1}{8}\Psi(\alpha(t))+\frac{1}{2}K_{R}\Psi\left(\alpha(t)\right)\right)
\end{eqnarray*}
Using (\ref{eq23}), we obtain
\begin{eqnarray*}
    A&=&\lambda_{1}, \; B=\lambda_{2}, \; C=\lambda_{3}, \; D=\lambda_{4}, \; F=\lambda_{5}\\
    \xi&=&\frac{\left(\lambda_{5}-
    4i\lambda_{2}\theta_{2}\dot{\theta}_{2}
    -4i\lambda_{3}+8\theta_{2}\log(x)-8\theta_{1}\log^{2}(x)
    -8\lambda_{2}\dot{\theta}_{2}\right) \Psi(\alpha(t))}{4\left(\beta_{5}+\log(x)\Psi(\alpha(t))\right)}\\
 &+&\frac{\left(-2\lambda_{2}\theta^{2}_{2}-
    8\lambda_{2}\theta^{2}_{1}-8\lambda_{3}\left(\theta_{2}+i\dot{\theta}_{2}\right)
 +8\lambda_{2}\dot{\theta}^{2}_{1}
 +2\lambda_{2}\dot{\theta}^{2}_{2}
  +4\lambda_{4}\theta_{2} \right)\Psi(\alpha(t))}{4\left(\beta_{5}+\log(x)\Psi(\alpha(t))\right)}\\
&+&\frac{\left( +4\lambda_{1}\log^{2}(x)
-8\lambda_{2}\dot{\theta}_{1}\dot{\theta}_{2}
 +8\lambda_{2}\theta_{2}\theta_{1}+4i\lambda_{4}\dot{\theta}_{2}
 -8i\lambda_{3}\dot{\theta}_{1}\log^{2}(x)
 \right)\Psi(\alpha(t))}{4\left(\beta_{5}+\log(x)\Psi(\alpha(t))\right)}\\
&&+\frac{\left(
  +8i\dot{\theta}_{2}\log(x)+8i\lambda_{2}\theta_{1}
-8\lambda_{4}\theta_{1}\log(x) +16\lambda_{3}\theta_{1}\log^{2}(x)\right)\Psi(\alpha(t))}{4\left(\beta_{5}+\log(x)\Psi(\alpha(t))\right)}\\
&&+\frac{\left( +8i\lambda_{2}\theta_{2}\dot{\theta}_{1}
+8i\lambda_{2}\theta_{1}\dot{\theta}_{1}
-16i\lambda_{2}\theta_{1}\dot{\theta}_{1}+
16i\lambda_{3}\dot{\theta}_{1}\log^{2}(x)\right)\Psi(\alpha(t))}{4\left(\beta_{5}+\log(x)\Psi(\alpha(t))\right)}\\
&+&\frac{-4\lambda_{4}\beta_{4}-4\lambda_{5}\beta_{6}
-4\lambda_{3}\beta_{3}-4\lambda_{2}\beta_{2}}{4\left(\beta_{5}+\log(x)\Psi(\alpha(t))\right)}\\
    K_{F}\Psi  (\alpha(t))&=&\beta_{1},\;K_{G}\Psi  (\alpha(t))=\beta_{2},\;K_{H}\Psi  (\alpha(t))=\beta_{3},\;K_{P}\Psi  (\alpha(t))=\beta_{4},\\
    K_{Q}\Psi  (\alpha(t))&=&\beta_{5},\;K_{R}\Psi  (\alpha(t))=\beta_{6}
\end{eqnarray*}
with
 $\lambda_{1}, \;\lambda_{2}, \; \lambda_{3}, \; \lambda_{4},
\;\lambda_{5},\;\beta_{1},\;\beta_{2},\;\beta_{3},\;\beta_{4},\;\beta_{5},\;\beta_{6}\in\mathrm{I\!R}
$.

We have  the following equation

\begin{eqnarray*}
    \frac{id
\Psi
(\alpha(t))}{dt}&=&\frac{1}{2}\left(-\lambda_{1}\log^{2}(x)-\lambda_{2}\left(2iz_{1}
    +\left(-iz_{2}+2iz_{1}\log(x)\right)^{2}\right)\right)\Psi(\alpha(t))\\
&+&\frac{1}{2}\left(\lambda_{3}\left(2\log(x)z_{2}-4z_{1}\log^{2}(x)+i\right)+\lambda_{4}\left(-z_{2}+2z_{1}\log(x)\right)\Psi(\alpha(t))\right)\\
&+&\frac{1}{2}\left(\xi(\theta,\dot{\theta},x)\log(x)-\frac{1}{4}\lambda_{5}\right)\Psi(\alpha(t))+\frac{1}{2}\left(\lambda_{1}K_{F}\Psi
(\alpha(t))+\lambda_{2}K_{G}\Psi  (\alpha(t))\right.\\
&+&\frac{1}{2}\lambda_{3}K_{H}\Psi
(\alpha(t))+\frac{1}{2}\left.\lambda_{4}K_{P}\Psi
(\alpha(t))+\frac{1}{2}\xi(\theta,\dot{\theta},x)K_{Q}\Psi
(\alpha(t))+\frac{1}{2}\lambda_{5}K_{R}\Psi  (\alpha(t))\right)
\end{eqnarray*}\end{proof}
We have the following theorem.\\
 Let\begin{eqnarray*}
  \Psi:  \mathbb{S}^{J}&\longrightarrow& L^{2}(\mathrm{I\! R}) \\
  \alpha(t)&\mapsto& \exp\left\{\frac{1}{2}\left(c(x)-iz_{2}\log(x)+iz_{1}\log^{2}(x)-\Phi(\theta)\right)\right\}
\end{eqnarray*} the exponential application define by lognormal
distribution.
\begin{theorem}\label{th2} Let $TS$ be the H$\ddot{a}$hler manifold.
 Let $\mathbb{S}^{J}$ be a Jacobi space and $\mathfrak{g}^{J}$ be the
  Lie algebra of the Jacobi group $G^{J}(\mathrm{I\! R})$.
  If the unique application
 $K_{L}\Phi(t):=\kappa\left(\alpha(t),L\right)$ is zero, then the
 equation (\ref{eq21}), the Hamiltonian equation on the Jacobi space $\mathbb{S}^{J}$, is given by
\begin{eqnarray*}
    \frac{id
\Phi(t)}{dt}&=&\frac{1}{2}\mathcal{H}(\theta,\dot{\theta},x)\Phi(t)\end{eqnarray*}
with
\begin{eqnarray*}
    \mathcal{H}(\theta,\dot{\theta},x)&=&-\lambda_{1}\log^{2}(x)-\lambda_{2}\left(2iz_{1}
    +\left(-iz_{2}+2iz_{1}\log(x)\right)^{2}\right)\\
&+&\lambda_{3}\left(2\log(x)z_{2}-4z_{1}\log^{2}(x)+i\right)+\lambda_{4}\left(-z_{2}+2z_{1}\log(x)\right)\\
&+&\xi(\theta,\dot{\theta},x)\log(x)-\frac{1}{4}\lambda_{5}.
\end{eqnarray*}
where
 \begin{eqnarray*}
    \xi(\theta,\dot{\theta},x)&=&
    \frac{\left(-4i\ddot{\theta}_{1}\log^{2}(x)
    -4i\log(x)\ddot{\theta}_{2}
     +4\lambda_{1}\log^{2}(x)
       +8i\lambda_{2}\dot{\theta}_{1}\theta_{2}+\lambda_{5}-4i\lambda_{3}\right)\varphi(t)}{4\left(\beta_{5}+\log(x)\right)\varphi(t)}\\
 &+&\frac{\left(+16\lambda_{3}\theta_{1}\log^{2}(x)
 -8\lambda_{4}\theta_{1}\log(x)
 +8\lambda_{2}\theta_{1}\theta_{2}
 +8i\lambda_{2}\theta_{1}+4i\lambda_{4}\dot{\theta}_{2}
-8\lambda_{2}\dot{\theta}_{2}\dot{\theta}_{1}\right)\Phi(t)}{4\left(\beta_{5}+\log(x)\right)\Phi(t)}\\
 &&+\frac{\left(2\lambda_{2}\dot{\theta}^{2}_{2}
 +8\lambda_{2}\dot{\theta}_{1}+4\lambda_{4}\theta_{2}
 -8\lambda_{2}\dot{\theta}_{1}-2\lambda_{2}\theta^{2}_{2}-8\lambda_{2}\theta^{2}_{1}
 -4\dot{\theta}_{1}\log^{2}(x)-4\dot{\theta}_{2}\log(x)\right)\Phi(t)}{4\left(\beta_{5}+\log(x)\right)\Phi(t)}\\
 &&+\frac{
 \left(-8\lambda_{3}\left(\theta_{2}+i\dot{\theta}_{2}\right)\log(x)
 +8\lambda_{2}\theta_{1}\dot{\theta}_{2}
 -16i\lambda_{2}\theta_{1}\dot{\theta}_{1}
 +16i\lambda_{3}\dot{\theta}_{1}\log^{2}(x)
 \right.}{4\left(\beta_{5}+\log(x)\right)\Phi(t)}\\
&&+\frac{
 \left.
 -8i\lambda_{4}\dot{\theta}_{2}
 -4i\lambda_{2}\theta_{2}\dot{\theta}_{2}\right)\Phi(t)}{4\left(\beta_{5}+\log(x)\right)\Phi(t)}\\
&+&\frac{-4\lambda_{1}\beta_{1}-4\lambda_{2}\beta_{2}
-4\lambda_{2}\beta_{2}
-4\lambda_{3}\beta_{3}-4\lambda_{4}\beta_{4}-4\lambda_{5}\beta_{6}
}{4\left(\beta_{5}+\log(x)\right)\Phi(t)}
\end{eqnarray*},
with
$z_{1}=\theta_{1}+i\dot{\theta}_{1},\;z_{2}=\theta_{2}+i\dot{\theta}_{2}$.
\end{theorem}

\begin{proof} Using equation (\ref{eq23}) and
\ref{pro4}, we have the result.
\end{proof}

We have the following theorem
\begin{theorem}\label{th1}
Let $TS$ be a K$\ddot{a}$hler manifold.  Let $\mathbb{S}^{J}$ be a
Jacobi space and $\mathfrak{g}^{J}$ be the Lie algebra of the Jacobi
group $G^{J}(\mathrm{I\! R})$. The equation (\ref{eq21}), is the
Shr$\ddot{o}$dinger equation on the Jacobi space $\mathbb{S}^{J}$ if
and only if $\xi(\theta,\dot{\theta},x)=\gamma, \; \gamma\in
\mathrm{I\! R}$, and we have \begin{eqnarray}\label{eq24}
    \frac{id
\Phi(t)}{dt}&=&\frac{1}{2}Q(L)\Phi(t)+\frac{1}{2}K(L)\Phi(t).
\end{eqnarray}, and furthermore, if $K_{L}\Phi(t):=\kappa\left(\alpha(t),L\right)$ is
zero, then the  Shr$\ddot{o}$dinger equation becomes
\begin{eqnarray*}
    \frac{id
\Phi(t)}{dt}&=&\frac{1}{2}\mathcal{H}\Phi(t)\end{eqnarray*} with
$\mathcal{H}=\frac{1}{2}Q(L)$, and the Hamiltonian is given by
\begin{eqnarray*}
    \mathcal{H}(\theta,\dot{\theta},x)&=&-\lambda_{1}\log^{2}(x)-\lambda_{2}\left(2iz_{1}
    +\left(-iz_{2}+2iz_{1}\log(x)\right)^{2}\right)\\
&+&\lambda_{3}\left(2\log(x)z_{2}-4z_{1}\log^{2}(x)+i\right)+\lambda_{4}\left(-z_{2}+2z_{1}\log(x)\right)\\
&+&\gamma\log(x)-\frac{1}{4}\lambda_{5}.
\end{eqnarray*}
\end{theorem}

\begin{proof}
By setting $\xi(\theta,\dot{\theta},x)=\gamma$, and using the
proposition \ref{pro4}, we have the result.

\end{proof}

\section{Conclusion}\label{sec6}
The aim of this work was to describe the variations of the spectral
curves of a point on the lognormal statistical manifold. Our
research question was: Is it possible to use Dombrowski's
constructions to describe the variations of a spectral curve
evolving in the Siegel Jacobi space of the lognormal statistical
manifold? to which we were able to show that we can put new
structures on the lognormal statistical manifold in order to
determine the Hamiltonian vector field flow of all spectral curves
on the Jacobi space. We showed that these variations can be
described by an equation, and under certain conditions this equation
is the Shr$\ddot{o}$dinger equation. This allowed us to determine
the mechanical energy of the evolution of the flow, known as the
Hamiltonian. We show that this Hamiltonian depends on the variable
associated with the evolution of the curves.

\section{General conclusion}\label{sec8}
The aim of this work was to describe the variations of the spectral
curves of a point on the lognormal statistical manifold. Our
research question was  is it possible to use Dombrowski's
constructions to describe the variations of a spectral curve
evolving in the Siegel-Jacobi space of the lognormal statistical
manifold? to which we were able to show that we can put new
structures on the lognormal statistical manifold in order to
determine the Hamiltonian vector field flow of all spectral curves
on the Jacobi space. We showed that these variations can be
described by an equation, and under certain conditions this equation
is the Shr$\ddot{o}$dinger equation. This allowed us to determine
the mechanical energy of the evolution of the flow, known as the
Hamiltonian. We show that this Hamiltonian depends on the variable
associated with the evolution of the curves.

 \backmatter

\bmhead{Supplementary information} This manuscript has no additional
data.

\bmhead{Acknowledgments}I would like to thank all the members of the
Mathematics Department of the University of Yaound$\acute{e}$1, and
all the members of the Algebra and Geometry Laboratory of the
University of Yaound$\acute{e}$1. We would like to thank Molitor
Mathieu \email{pergame.mathieu@gmail.com} of Instituto de
Matem$\acute{a}$tica, Universidade Federal da Bahia, Av. Adhemar de
Barros, S/N, Ondina, pergame.mathieu@gmail.com, Ondina, 40170-110,
Salvador, BA, Brazil.

\section*{Declarations}
This article has no conflict of interest to the journal. No
financing with a third party.
\begin{itemize}
\item No Funding
\item No Conflict of interest/Competing interests (check journal-specific guidelines for which heading to use)
\item  Ethics approval
\item  Consent to participate
\item  Consent for publication
\item  Availability of data and materials
\item  Code availability
\item Authors' contributions
\end{itemize}
\bibliography{bibliography}

\end{document}